\documentclass[11pt,a4paper]{amsart}
\usepackage[utf8]{inputenc}
\usepackage{bbm,amsmath,amssymb,amsfonts,graphicx,color,mathtools, verbatim}
\usepackage[parfill]{parskip}
\usepackage[hidelinks]{hyperref}
\usepackage{thmtools, thm-restate}
\usepackage[size=footnotesize]{caption}
\usepackage[size=footnotesize]{subcaption}

\usepackage{mathrsfs}

\usepackage{xcolor}

\newcommand{\abb}[5]{%
\setlength{\arraycolsep}{0.4ex}%
\begin{array}{rcccc}%
#1 &:\,& #2 & \,\,\longrightarrow\,\, & #3 \\[0.5ex]%
     & & #4 & \longmapsto & #5%
\end{array}%
}

\newcommand{\N}{\mathbb{N}}
\newcommand{\R}{\mathbb{R}}
\newcommand{\X}{\underline{X}}
\newcommand{\C}{\mathcal{C}}
\renewcommand{\P}{\mathcal{P}}
\renewcommand{\H}{\mathcal{H}}
\renewcommand{\L}{\mathcal{L}}
\newcommand{\V}{\mathcal{V}}
\newcommand{\W}{\mathcal{W}}
\newcommand{\E}{\mathcal{E}}

\newcommand{\Sym}{\mathcal{S}\text{ym}}
\newcommand{\sign}{\text{sgn}}

\renewcommand{\emph}{\textbf}

\DeclarePairedDelimiter\ceil{\lceil}{\rceil}
\DeclarePairedDelimiter\floor{\lfloor}{\rfloor}

\newtheorem{theorem}{Theorem}[section]
\newtheorem{lemma}[theorem]{Lemma}
\newtheorem{corollary}[theorem]{Corollary}
\newtheorem{proposition}[theorem]{Proposition}
\newtheorem{remark}[theorem]{Remark}
\newtheorem{example}[theorem]{Example}
\newtheorem{definition}[theorem]{Definition}

\newtheorem{conjecture}[theorem]{Conjecture}
\newtheorem{question}[theorem]{Question}

\title{Shellable slices of hyperbolic polynomials and the degree principle}
\author{Arne Lien}
\address{Department of Mathematics and Statistics, UiT - the Arctic University of Norway, 9037 Troms\o, Norway}
\email{arne.lien@uit.no}

\author{Robin Schabert}
\address{Department of Mathematics and Statistics, UiT - the Arctic University of Norway, 9037 Troms\o, Norway}
\email{robin.schabert@uit.no}
\date{}

\thanks{This work has been supported by Tromsø Research Foundation under the grant agreement 17matteCR (SymRAG) and the European Union’s Horizon 2020 research and innovation programme under the Marie Sklodowska-Curie Actions, grant agreement 813211
(POEMA)}

\begin{document}

\maketitle

\begin{abstract}
    We study a natural stratification of certain affine slices of univariate hyperbolic polynomials. We look into which posets of strata can be realized and show that the dual of the poset of strata is a shellable simplicial complex and in particular a combinatorial sphere. From this we obtain a g-theorem and an upper bound theorem on the number of strata. We use these results to design smaller test sets to improve upon Timofte's degree principle and give bounds on how much the degree principle can be improved.
\end{abstract}

Univariate polynomials with only real roots are called \emph{hyperbolic polynomials}. We will study \emph{hyperbolic slices}, that is, sets of hyperbolic polynomials that share the same first few coefficients. We stratify these sets in terms of the arrangements and multiplicities of the roots of the hyperbolic polynomials and then we study the combinatorial structure of the poset of strata and its implications for the study of real symmetric varieties.

These hyperbolic slices have a rich geometric structure that has been studied by several authors. For instance,  \cite{arnol1986hyperbolic}, \cite{givental1987moments} and \cite{kostov1989geometric} studied \emph{Vandermonde varieties}, that is, varieties given by the first few elementary symmetric polynomials. Since the set of monic hyperbolic polynomials can be viewed as the orbit space of the symmetric group and the elementary symmetric polynomials generate the ring of symmetric polynomials, Vandermonde varieties are the fibers of hyperbolic slices. Thus, hyperbolic slices are deeply connected to Vandermonde varieties.

More generally, hyperbolic slices are not just connected to Vandermonde varieties, but to the study of any symmetric variety. In \cite{riener2012degree} and \cite{riener2024linear} this connection is exploited to prove and generalize Timofte's degree and half-degree principle for the symmetric group. The degree principle implies that symmetric polynomials of degree at most $d$ have a common real root if and only if they have a common real root with at most $d$ distinct coordinates, thus it allows one to show nonemptyness of symmetric varieties much faster than arbitrary varieties. We study the poset of strata of hyperbolic slices in order to make improvements on this degree principle.

The poset of strata of hyperbolic slices was already studied by one of the authors in \cite{lien2023hyperbolic} and the question was raised if it is polytopal. We are able to show the weaker statement in Theorem \ref{3:thm:shell} that the dual of the poset of strata is generically a shellable simplicial complex. To prove this, we will first generalize a result by Arnold \cite{arnol1986hyperbolic} and Meguerditichian \cite{meguerditchian1992theorem}. They show that every hyperbolic slice has a unique minimal and maximal polynomial with respect to the first free coefficient and that these polynomials are generically uniquely characterized by alternating single and multiple roots. We show in Theorem \ref{2:thm:minmax} that an analogous result is true for every stratum of a hyperbolic slice. Then we use this to show that the dual poset is generically a shellable simplicial complex and therefore a combinatorial sphere (Corollary \ref{3.cor:sph}). From this, we obtain the same bounds and relations on the number of $i$-dimensional strata as for certain polytopes. Namely, we obtain a "$g$-theorem" (Corollary \ref{3:cor:g-thm}) for generic hyperbolic slices and an "upper bound theorem" (Corollary \ref{3:cor:ubt}) for the general case.

With the connection between hyperbolic slices and real symmetric varieties, we can use these combinatorial results to improve upon Timofte's degree principle. The degree principle allows one to show nonemptyness of real symmetric varieties by reducing the number of variables needed to the minimal amount, and so our improvement lies in reducing the number of orbit types needed to check. Thus we improve the degree principle by considering test sets that have smaller sizes. These test sets, which we call \emph{Vandermonde coverings}, are therefore characterized by certain orbit types of the symmetric group. We give a lower and an upper bound on the size of an optimal Vandermonde covering (Theorem \ref{4:upper_bound} and \ref{4:recursive_lower_bound}) and outline a computational approach on how to get better and maybe optimal Vandermonde coverings for real symmetric varieties given by polynomials in few variables and low degrees.

We conclude with several open questions and conjectures on the stratification of hyperbolic slices.

\subsection*{Acknowledgements.}
We would like to thank Philippe Moustrou for his valuable comments on the manuscript.

\section{Preliminaries}
\subsection{Simplicial complexes, shellings and spheres}

\begin{definition}
    A \emph{poset} $(P,\leq)$, or \emph{partially ordered set}, is a set $P$ equipped with a partial order $\leq$.
\end{definition}

We usually just write $P$ if the partial order is clear from context. Also, we say that an element $a$, of a poset, $P$, \emph{covers} $b\in P$ if $b\leq a$ and for any $c\in P$ with $b\leq c\leq a$, we have $c=a$ or $c=b$. So we see that the partial order on $P$ is \textbf{generated} by its covering relations in the following sense: let $a,b\in P$, then $a\leq b$ if there is a sequence of elements $c_1,\dots,c_m$ with $c_1=a$, $c_m=b$ and where $c_i$ is covered by $c_{i+1}$ for any $i\in [m-1]$.

Next, we say that two posets $(P,\leq)$ and $(Q,\leq^*)$ are \emph{isomorphic} if there exists an order-preserving bijection between $P$ and $Q$. %$\phi:P\to Q$ such that $a\leq b$ if and only if $\phi(a)\leq^*\phi(b)$ for all $a,b\in P$. 
Lastly, we say that the poset $(P,\geq)$ is the \emph{dual} poset of $(P,\leq)$. 
%That is, the dual of $(P,\leq)$ is the set $P$ equipped with the converse partial order $\leq^*$, given by $a\leq^* b$ if and only if $b\leq a$.

An important subclass of posets are simplicial complexes.

\begin{definition}
    A \emph{simplicial complex} is  a family of finite sets that is closed under taking subsets. A \emph{geometric simplicial complex} is a family of simplices, $S$ in $\R^m$, such that each face of a simplex in $S$ is also in $S$ and such that the intersection of two simplices is a face of each simplex.
\end{definition}

Thus any simplicial complex may be identified with a family, $C$, of subsets of $[m]:=\{1,2,\dots,m\}$ for some nonnegative integer $m$, such that if $A\subset B\in C$, then $A\in C$. A \emph{geometric realization} of $C$ is a geometric simplicial complex $S$ whose poset of simplices is isomorphic to $C$. Since all simplicial complexes have a geometric realization, we will usually not distinguish between a geometric realization and the simplicial complex. Instead, it should always be clear from the context which object we are referring to.

We can construct a geometric realization of a simplicial complex $C$ by identifying the smallest nonempty sets of $C$ with the points $e_1,\dots,e_m\in\R^m$, where $e_i$ is the $i^{th}$ standard basis vector, and then take the convex hull of $e_{i_1},\dots,e_{i_k}$ whenever $\{i_1,\dots,i_k\}$ is an element of $C$.

Just like for a geometric realization of $C$, the elements of $C$ are called \emph{faces}. The \emph{dimension} of a face is defined as the dimension of the corresponding face in a geometric realization and the dimension of $C$ is the dimension of its highest-dimensional faces. Also, the faces that are maximal with respect to inclusion are called \emph{facets}, the second largest are called \emph{ridges} and the smallest nonempty faces are called \emph{vertices}. When all the facets have the same dimension, the simplicial complex is called \textbf{pure}.

We also need this natural generalization of a geometric simplicial complex:
\vspace{-0.2in}
\begin{definition}
    A \emph{polytope complex} is a family of polytopes $C$, in $\R^m$, such that each face of a polytope is in $C$ and such that the intersection of two polytopes is a face of each.
\end{definition}

As with simplicial complexes, we will usually not distinguish between a polytope complex and its abstract poset of polytopes.

Lastly, we need to talk about a particular class of polytope complexes that are similar to spheres from a combinatorial point of view. Note that a \emph{simplicial sphere} is a geometric simplicial complex which is homeomorphic to a sphere. But showing that a simplicial complex is a simplicial sphere can be difficult and thus we introduce the so-called "combinatorial spheres".

\begin{definition}
    A \emph{subdivision} of a polytope complex $C$ is a polytope complex $S$ such that $$\bigcup_{I\in S}I=\bigcup_{J\in C}J\subset \R^m$$ and such that each face of $S$ is contained in a face of $C$. Moreover, we say a subdivision $S$ is \emph{simplicial} if $S$ is a geometric simplicial complex.
\end{definition}

\begin{definition}
    A \emph{combinatorial (or PL) $m$-sphere} is a polytope complex for which there exists a simplicial subdivision which is isomorphic to a simplicial subdivision of the boundary of a $(m+1)$-dimensional simplex.
\end{definition}

To determine if a simplicial complex is a combinatorial sphere, we need the notion of shellability.

\begin{definition}
    A \emph{shelling} of a pure simplicial complex, $C$, is an ordering of the facets, $F_1,\dots,F_k$, such that for any $i\in\{2,..,k\}$, the simplicial complex 
    $$\bigcup_{j=1}^{i-1}F_j\cap F_i$$ is pure of dimension $\dim(C)-1$. If there exists a shelling of $C$, then $C$ is called \emph{shellable}.
\end{definition}

Then from Proposition 1.2 in \cite{danaraj1974shellings} we have the following result:

\begin{proposition}\label{prelim:plsphere}
    A shellable simplicial complex of dimension $m$, whose ridges are all contained in exactly two facets, is a combinatorial $m$-sphere.
\end{proposition}

\subsection{Symmetric polynomials and Vandermonde varieties}

Throughout the article, we denote by $\Sym(n)$ the symmetric group on the set $[n]$, $\R[\X]:=\R[X_1,\dots,X_n]$ the polynomial ring in $n$ variables over $\R$ and by $\R[\X]^{\Sym(n)}$ the subring of symmetric polynomials.

\begin{definition}\label{0:def:el.sym.powersum}
For $i\in [n]$, we denote by
\[E_i:=\sum_{1\leq j_1<\dots<j_i\leq n}X_{j_1}\cdots X_{j_i}\] 
the $i^{th}$ \emph{elementary symmetric polynomial} and by
\[P_i:=\sum_{j=1}^n X_j^i\] 
the $i^{th}$ \emph{power sum}.
\end{definition}

The Fundamental Theorem of Symmetric Polynomials states, that every polynomial can be uniquely written in terms of the elementary symmetric polynomials. Furthermore, we have the following:

\begin{theorem}[Fundamental Theorem of Symmetric Polynomials]\label{4:lem:cordian}
Any symmetric polynomial $F\in \R[\X]^{\Sym(n)}$ of degree $s$, with $s\leq n$, can be uniquely written as
\[H=G(E_1,\dots,E_s),\]
where $G$ is a polynomial in $\R[Z_1,\dots,Z_s]$.
\end{theorem}
\begin{proof}
Proposition 2.3 in \cite{riener2012degree}.
\end{proof}

Theorem \ref{4:lem:cordian} is a key tool in the proof of the degree principle in \cite{riener2012degree}.

\begin{theorem}[Degree principle]\label{0:thm:degree}
Let $f_1,\dots,f_k\in \R[\X]^{\Sym(n)}$ be symmetric polynomials of degree at most $d< n$. Then the real variety
\[V_\R(f_1,\dots,f_k)\]
is nonempty if and only if it contains a point with at most $d$ distinct coordinates.
\end{theorem}

\begin{definition}\label{0:def:composition}
A sequence of positive integers $\mu=(\mu_1,\ldots,\mu_l)$ which sum up to $n$ is called a \emph{composition of $n$ into $l$ parts} and we call $\ell(\mu):=l$ the \emph{length} of $\lambda$. 
\end{definition}

Next, we introduce Vandermonde varieties and the Weyl chamber:
\begin{definition}\label{0:def:vandermonde}
For $s\in [n]$ and $a\in \R^s$, we call
\[\V(a) := \left\{ x\in \R^n ~\middle|~ -E_1(x)=a_1,\dots,(-1)^sE_s(x)=a_s \right\}\]
the \emph{Vandermonde variety} of $a$. For a monic polynomial
\[F=T^n+F_1T^{n-1}+\dots+F_n\]
and $s\leq n$, we define
\[\V_s(F):=\V(F_1,\dots,F_s).\]
Furthermore, for a composition $\mu$ of $n$ and a polynomial $Q\in \R[\X]$ we define
\[Q^\mu:=Q^\mu(\underbrace{X_1,\dots,X_1}_{\mu_1-\text{times}},\underbrace{X_2,\dots,X_2}_{\mu_2-\text{times}},\dots,\underbrace{X_s,\dots,X_s}_{\mu_s-\text{times}})\in\R[X_1,\dots,X_l]\]
and
\begin{align*}
    \V_s^\mu(F) := \{x\in\mathbb{R}^l ~|~ (-1)^iE_i^\mu(x) = F_i \ \forall \ i\in [s]\}
\end{align*}
the \emph{Vandermonde variety} of $F$ with respect to $\mu$ and $s$.
\end{definition}

\begin{definition}\label{0:def:weyl-chamber}
For $l\in \N$, we denote by
\[\W_l:=\left\{ x\in \R^l ~\middle|~ x_1\leq\dots\leq x_l \right\}\]
the $l$-dimensional \emph{Weyl chamber}.
\end{definition}

\section{Hyperbolic slices and posets}\label{sec:strat}

Throughout the article, we will denote by $\H\subset \R[T]$ the set of monic hyperbolic polynomials, that is, the monic polynomials with only real roots. Furthermore, we fix a monic hyperbolic polynomial $F\in \H$ of degree $n\in\N$ and an integer $s\in\N$, with $s\leq n$. Then the sets of hyperbolic polynomials that we will study are the following.

\begin{definition}
     We call the affine slice 
     \[\H_s(F) = \{T^n+H_1T^{n-1}+\dots + H_n \in \H ~|~ H_i=F_i \ \forall \ i\in [s]\},\] 
     where $F=T^n+F_1T^{n-1}+\dots +F_n$, a \emph{hyperbolic slice}.
\end{definition}

First, we recall some previously established results on hyperbolic slices and provide examples of hyperbolic slices and their stratifications. In particular, we will see that the strata are contractible and we see a characterization of the strata's relative interior and the closure of their relative interior.

Then we introduce a generalization of the main theorem in \cite{meguerditchian1992theorem}. In that article, they investigate the following question: for which monic hyperbolic polynomials $H$, of degree $n$, is $H+c_0T^k+\dots+c_k$ not hyperbolic for any $c_0,\dots,c_k\in\R$ with $c_0> 0$ (resp. $c_0< 0$) and $k<n$?  They call such polynomials "$k$-maximal" (resp. "$k$-minimal") and characterize which polynomials are $k$-minimal and $k$-maximal. Thus they characterize which polynomials in $\H_s(F)$ have a minimal first free coefficient and which polynomials have a maximal one. We extend this question to the strata of hyperbolic slices and prove an analogous result.

\subsection{Stratification of hyperbolic slices}

We will study a particular stratification of $\H_s(F)$ and in order to define this stratification, we need to introduce a partial order on compositions.

\begin{definition}
For two compositions of $n$, $\mu$ and $\lambda$, we let $\mu\leq\lambda$ if there is a composition $\nu$ of $\ell(\lambda)$  of length $l=\ell(\mu)$ such that $$\mu= (\lambda_1+\dots+\lambda_{\nu_1},\dots,\lambda_{\ell(\lambda)-\nu_l+1}+\dots+\lambda_{\ell(\lambda)}).$$
\end{definition}

In other words $\mu\leq\lambda$ if one can obtain $\mu$ from $\lambda$ by replacing some of the commas in $\lambda$ with plus signs. For a hyperbolic polynomial $H$ with distinct roots $b_1<\dots < b_l$ and respective multiplicities $m_1,\dots, m_l$ we will let $c(H) = (m_1,\dots,m_l)$ denote the \emph{composition of $H$}. 

\begin{definition}
    Let $\mu$ be a composition of $n$. Then we define the stratum
    \[\H_s^\mu(F):=\{H\in\H_s(F)~|~c(H)\leq \mu\},\]
    of $\H_s(F)$, and we call the poset of strata of $\H_s(F)$, partially ordered by inclusion, a \emph{hyperbolic poset} and denote it by $\L_s(F)$.
\end{definition}

We commonly identify monic polynomials of degree $n$ in $\R[T]$ with points in $\R^n$. Thus we will be equipping $\H_s^\mu(F)$ with the subspace topology of the Euclidean topology on $\R^n$.
\begin{remark}\label{2:rem:vieta}
The set $\H$ of hyperbolic polynomials can be seen as the image of the \emph{Vieta map}
\[\abb{\E}{\R^n}{\H}{x}{(-E_1(x),\dots,(-1)^nE_n(x)).}\]
Moreover, $\E$ maps the Vandermonde variety intersected with the Weyl chamber $\V(F_1,\dots,F_s)\cap \W_n$ homeomorphically (see Lemma 2.1 in  \cite{lien2023hyperbolic}) to the hyperbolic slice $\H_s(F)$. So a stratum $\H_s^\mu(F)$ is homeomorphic to
\[\big\{ (\underbrace{x_1,\dots,x_1}_{\text{$\mu_1$-times}},\dots,(\underbrace{x_l,\dots,x_l}_{\text{$\mu_l$-times}}) ~\big|~ x_1,\dots,x_l)\in\R^l \big\}\cap \V(F_1,\dots,F_s)\cap \W_n\]
under the Vieta map.
\end{remark}
Since $\H_s^\mu(f)$ is the image of a polyhedron intersected with a real algebraic set defined by $s$ polynomials, then in accordance with the terminology in real algebraic geometry, we call $\H_s^\mu(F)$ \emph{generic} if it contains no polynomial with at most $s-1$ distinct roots.

Note that not all compositions need to occur in $\H_s(F)$ and two distinct compositions do not necessarily give rise to distinct strata as can be seen in the following examples:
\begin{example}\label{1:ex:pyramid}
    Let $n=6$, $s=3$ and let 
    \begin{alignat*}{3}
    G&:=T^6 - \frac{21}{4} T^4 + &&T^3 &&+\frac{21}{4}T^2 - 1 \quad \text{and}\\
    H&:=T^6 - \frac{21}{4} T^4  && &&+\frac{21}{4}T^2 - 1.
    \end{alignat*}
    Consider the hyperbolic slices $\H_3(G)$ and $\H_3(H)$. One can label the strata of these hyperbolic slices by the corresponding compositions as exemplified for the $0$-dimensional strata of $\H_3(H)$ in Figure \ref{non-generic_labeled.png}. The other strata of $\H_3(H)$ can be labeled similary, e.g. the polynomials on the blue curve between $(1,4,1)$ and $(3,3)$ have corresponding composition $(1,2,2,1)$. Note that $\H_3(H)$ is non-generic while $\H_3(G)$ is generic.
    \begin{figure}[ht]
 \centering
 \begin{subfigure}[b]{0.49\textwidth}
   \centering
   \includegraphics[width=\textwidth]{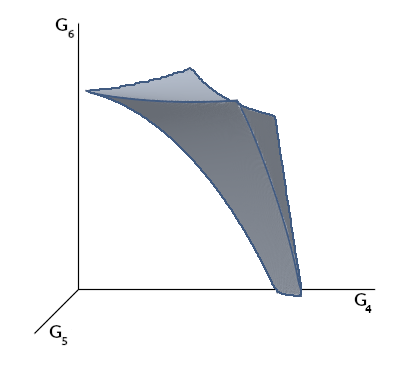}
   \caption{$\H_3(G)$}\label{generic.png}
 \end{subfigure}
  \hfill
 \begin{subfigure}[b]{0.49\textwidth}
   \centering
   \includegraphics[width=\textwidth]{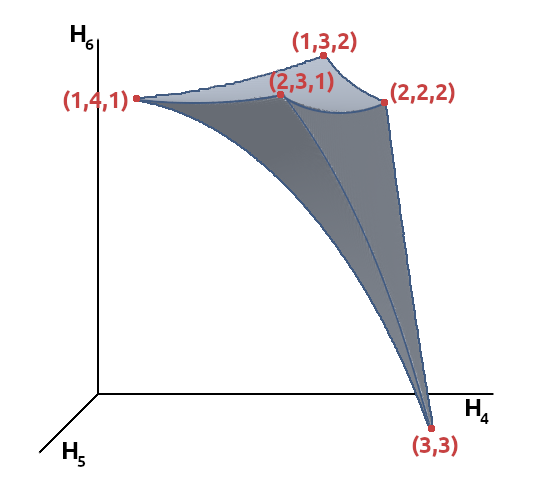}
   \caption{$\H_3(H)$}\label{non-generic_labeled.png}
 \end{subfigure}
 \hfill
 \caption{}
\end{figure}
\end{example}

From the examples, it looks like the strata have some nice geometric and combinatorial properties. We will present some of these geometric properties in a moment, but first note how the pictures are reminiscent of polytopes except that the strata are not convex. Thus it is natural to ask if this stratification of hyperbolic slices is always polytopal, that is, whether or not the hyperbolic poset is isomorphic to the face lattice of a polytope. We will not be able to answer this question, but we leave it as a conjecture.

\begin{conjecture}\label{2:conj:polyt}
    Hyperbolic posets are polytopal.
\end{conjecture}

From the example above we see that the poset of strata $\L_s(H)$ is isomorphic to the face lattice of a pyramid. However, one can check that there is no hyperplane containing the four polynomials with composition $(1,4,1), (2,3,1), (1,3,2)$ and $(2,2,2)$ even though they are all contained in a two-dimensional stratum. Thus $\L_s(H)$ is not poset isomorphic to the face lattice of the convex hull and the convex hull is therefore not the right candidate to show polytopality in general.

As mentioned we will not be answering Conjecture \ref{2:conj:polyt} in this article. Instead we will show that hyperbolic posets possess certain traits that are similar to polytopes. For instance, we will show in the next section that the dual of $\L_s(F)$ satisfies the Upper Bound Theorem in general and the g-Theorem in the generic case.

\begin{lemma}\label{2:lem:compcontr}
    The stratum $\H_s^\mu(F)$ is contractible or empty and when $s\geq 2$ it is compact.
\end{lemma}

\begin{proof}
     See Theorem 1.1 in \cite{kostov1989geometric} which was rephrased to our setting in \cite{lien2023hyperbolic}, see Proposition 2.2 and Lemma 3.2.
\end{proof}

The fact that the strata are contractible has some useful implications on how the compositions are distributed in $\H_s(F)$. To talk about these, note that as a consequence of Remark \ref{2:rem:vieta} $\H_s^\mu(F)$ is a semi-algebraic set, thus when we speak about the dimension of $\H_s^\mu(F)$, it is its dimension as a semi-algebraic set.

\begin{definition}
    Let $\H^\mu_s(F)$ be a nonempty stratum of dimension $d$, then
    \begin{enumerate}
        \item the \emph{relative interior} of $\H_s^\mu(F)$ is the set of polynomials $H\in \H_s^\mu(F)$ such that an open neighbourhood of $H$ is homeomorphic to an open set in $\R^d$ and
        \item the \emph{relative boundary} of $\H_s^\mu(F)$ is the set of polynomials $\H_s^\mu(F)$ that are not in the relative interior. 
    \end{enumerate}
\end{definition}

\begin{proposition}\label{2:prop:pointormaxd}
    Suppose the stratum $\H^\mu_s(F)$ contains a polynomial with at least $s$ distinct roots, then 
    \begin{enumerate}
        \item\label{2:prop:item1} the dimension of $\H^\mu_s(F)$ is $l-s$,
        \item\label{2:prop:item2} its relative interior is $\{H\in\H^\mu_s(f)~|~c(H)=\mu\}$ and
        \item\label{2:prop:item3} it equals the closure of its relative interior.
    \end{enumerate}
    If it contains no polynomial with at least $s$ distinct roots, then the stratum is either a single polynomial or empty.
\end{proposition}

\begin{proof}
    See Proposition 2.2, Theorem 2.6, Theorem 2.7 and Corollary 2.8 in \cite{lien2023hyperbolic}.
\end{proof}

\subsection{Escaping hyperbolic strata} In this subsection, we ask which polynomials of a stratum $\H_s^\mu(F)$ have a minimal (resp. maximal) first free coefficient. This was asked and answered for $\H_s(F)$ in \cite{meguerditchian1992theorem} and it turned out that the question could be fully answered by looking at the composition of the minimal (resp. maximal) polynomials. Thus they classified which polynomials in $\H_s(F)$ have the maximal first free coefficient and which have the minimal (when such polynomials exist). We shall give a similar classification, except we will restrict the domain to be any of the strata of $\H_s(F)$.

\begin{definition}
    We call $H=T^n+H_1T^{d-1}+\dots+H_d\in\H_s^\mu(F)$ a \emph{minimal} (resp. \emph{maximal}) polynomial of the stratum $\H_s^\mu(F)$ if $H_{s+1}\leq G_{s+1}$ (resp. $H_{s+1}\geq G_{s+1}$) for all $G=T^n+G_1T^{d-1}+\dots+G_d\in\H_s^\mu(F)$.
\end{definition}

As all the polynomials in $\H_s^\mu(F)$ will have an $i^{th}$ root of multiplicity at least $\mu_i$, it will be useful to mod out these multiplicities. Also, note that if a composition $\lambda$ is less than or equal to $\mu$, there is a unique composition $\nu$ such that $\lambda=(\mu_1,\dots,\mu_{\nu_1},\dots,\mu_{\nu_{l-1}+1},\dots,\mu_{\nu_l})$. Thus we define the following compositions:

\begin{definition}\label{2:def:mod}
    If $\lambda\leq\mu$, let $\lambda/\mu$ denote the composition $\nu$ such that $\lambda=(\mu_1,\dots,\mu_{\nu_1},\dots,\mu_{\nu_{l-1}+1},\dots,\mu_{\nu_l})$.
\end{definition}

To state the result, note that refer to the composition $\mu=(\mu_1,\mu_2,\dots,\mu_l)$ as an \emph{alternate odd} composition if $\mu_{l}=\mu_{l-2}=\dots=1$ and as an \emph{alternate even} composition if $\mu_{l-1}=\mu_{l-3}=\dots=1$.

\begin{restatable}{theorem}{minmax}\label{2:thm:minmax}
    Let $\lambda$ be the composition of $H\in \H^\mu_s(F)$ and let $s\geq 2$, then
    \begin{enumerate}
        \item\label{2:item:minmax1} there is a unique minimal (resp. maximal) polynomial in $\H^\mu_s(F)$ and
        \item\label{2:item:minmax2} the polynomial $H$ is minimal (resp. maximal) if and only if $\ell(\lambda)\leq s$ and $\lambda/\mu$ is less or equal to an alternate odd (resp. even) composition.
    \end{enumerate}
\end{restatable}

When $s=1$ there is also a maximal polynomial for all strata, but no minimal polynomial for any strata. The maximal polynomial is then the unique polynomial with only one distinct root and it thus follows from \cite{meguerditchian1992theorem}.

Note that in the generic case, one can replace $\lambda/\mu$ being "less than or equal" by "equal" in the above theorem since no two compositions of the same length are comparable. The proof of Theorem \ref{2:thm:minmax} is based on many of the same ideas as in \cite{meguerditchian1992theorem} and \cite{arnol1986hyperbolic}, however some of their techniques do not work in this general setting and others need to be adjusted. Thus the proof is rather lengthy and technical and has therefore been relocated to the appendix. Since the theorem is an important tool for this article, we will prove the first part here and sketch the proof of the second part. However, the interested reader may wish to skip ahead to the appendix after the proof of the first part.

We start by proving the first item and we will let $l=\ell(\mu)> s$ for the proof as $\H^\mu_s(F)$ is either empty or a point if $l\leq s$ according to Proposition \ref{2:prop:pointormaxd}.

\begin{lemma}\label{2:lem:projhom}
    The map 
    \[\abb{P^{n-l}}{\H_s^\mu(F)}{\R^{l-s}}{T^n+H_1T^{n-1}+\dots+H_n}{(H_{s+1},\dots,H_l)}\]
    is a homeomorphism onto its image and the image is closed in $\R^{l-s}$.
\end{lemma}

\begin{proof}
    See Proposition 2.5 in \cite{lien2023hyperbolic}.
\end{proof}

\begin{proof}[Proof of Item \ref{2:item:minmax1} from Theorem \ref{2:thm:minmax}]
    The statement is clear when $\H_s^\mu(F)$ is just a point so we will assume $\H_s^\mu(F)$ is $(l-s)$-dimensional. By Lemma \ref{2:lem:compcontr}, $\H_s^\mu(F)$ is compact so the existence of minimal and maximal polynomials is clear.
    \vspace{4mm}\\ 
    Let $H\in \H_s^\mu(F)$ be a minimal polynomial. To show uniqueness, we assume that $\H_{s+1}^\mu(H)$ contains another polynomial, i.e. it is of dimension $l-s-1$. By Proposition \ref{2:prop:pointormaxd}, it contains a polynomial $G$ with composition $\mu$. By Lemma \ref{2:lem:projhom} and Proposition \ref{2:prop:pointormaxd}, $P^{n-l}(\H_s^\mu(F))$ is full-dimensional with interior points corresponding to the image of the polynomials with composition $\mu$. This contradicts $G$ being minimal in $\H_s^\mu(F)$ as interior points of $P^{n-l}(\H_s^\mu(F))$ cannot have a minimal first coordinate if it is at least one dimensional. The argument for maximal polynomials is analogous.
\end{proof}

Next, we sketch the proof of the second part of Theorem \ref{2:thm:minmax} for generic hyperbolic slices. The proof is done by induction on the dimension of the strata of $\H_s(F)$. So the first step is to establish that the minimal (resp. maximal) polynomial, $H$ of $\H_s^\mu(F)$, has $s$ distinct roots and is such that $c(H)/\mu$ is alternate odd (resp. even) when $\H_s^\mu(F)$ is one-dimensional. This part of the proof is similar to some of the arguments in \cite{meguerditchian1992theorem}.

Firstly, note that a one dimensional $\H_s^\mu(F)$ is compact by Lemma \ref{2:lem:compcontr} and so it has two relative boundary points. Since relative interior points cannot be minimal or maximal, one of the two boundary points will be the minimal polynomial and the other the maximal polynomial. Also, from Proposition \ref{2:prop:pointormaxd} we get that the relative boundary of $\H_s^\mu(F)$ are polynomials with at most $s$ distinct roots.

To show that a polynomial $H=\prod_{i=1}^{s+1}(T-x_i)^{\mu_i}\in \H_s^\mu(F)$ is such that $c(H)/\mu$ is alternate odd if it is minimal and alternate even if it is maximal we use the theory of Lagrange multipliers. By looking at the Lagrangian function of our system of equations, we show that the only points in $\V_s^\mu(F)$ that are local extrema of the function $(-1)^{s+1}E_{s+1}$, are the points with at most $s$ distinct roots (see the discussion preceding Lemma \ref{lem:princminor}).

Then differently from the argument in \cite{meguerditchian1992theorem}, we are optimizing over a semi-algebraic set, not an algebraic one, so we have to show that optimizing over $\V_s^\mu(F)$ and $\V_s^\mu(F)\cap \W_{s+1}$ are equivalent in our setting. But it turns out that if $x=(x_1,...,x_{s+1})$ is a local extrema of $(-1)^{s+1}E_{s+1}$ on $\V_s^\mu(F)$, then $x$ is only close to tuples of roots corresponding to strata in $\H_s(F)$ for which $H$ is a minimal (resp. maximal) polynomial, that is, not both. Therefore we can use the Hessian criterion to characterize which of the two boundary points of $\H_s^\mu(F)$ is the minimal and maximal polynomial (see Proposition \ref{2:prop:indstart}).

Having established the theorem for the one-dimensional strata we move on to describing the induction step, which is also quite different from the argument in \cite{meguerditchian1992theorem}. So let $\H_s^\mu(F)$ be a stratum that is at least two-dimensional, then the key is firstly the following observation:

\begin{restatable*}{lemma}{maxiffmax}\label{2:lem:maxiffmax}
Let $l\geq s+2$, then the polynomial $H\in H_s^u(f)$ is minimal (resp. maximal) if and only if it is minimal (resp. maximal) for all strata that contain $H$ and that are strictly contained in $\H_s^\mu(F)$.
\end{restatable*}

Since $\H_s^\mu(F)$ is compact and of dimension at least $m\geq 2$, then by using Lemma \ref{2:lem:compcontr} and Proposition \ref{2:prop:pointormaxd} we can argue that the stratum contains at least two $(m-1)$-dimensional strata correspondin to two compositions, $\nu$and $\gamma$, of length $\ell(\mu)-1$. Next, we show that $c(H)/\mu$ is alternate odd (resp. even) if and only if $c(H)/\nu$ and $c(H)/\gamma$ are alternate odd (resp. even) (see Proposition \ref{2:prop:generic}).

Putting those two observations together let us go from having established the theorem for $k$-dimensional strata to establishing it for $(k+1)$-dimensional strata. Finally, extending the theorem to non-generic hyperbolic slices is done by perturbing the polynomial $F$ slightly (see Lemma \ref{2:lem:nongen} and Lemma \ref{2:lem:nongen2}). This works since, as we saw in Proposition \ref{2:prop:pointormaxd}, the strata equals the closure of their relative interior and since the relative interior of $\H_s^\mu(F)$ consists of the polynomials with composition $\mu$.

\section{Shellability of the dual poset}

We start in the first subsection by showing that in the generic case, the boundary complex of the dual of $\L_s(f)$ is a simplicial complex. Next, we use the results from the previous section to imitate a line shelling for polytopes thus showing that, in the generic case, the boundary complex of the dual poset is shellable and therefore a combinatorial $(d-s)$-sphere.

This has several consequences for both generic and non-generic hyperbolic slices. Thus in the second subsection, we can make use of the Upper Bound Theorem (UBT) and the g-theorem for simplicial spheres to get bounds on the number of $i$-dimensional strata in our poset.

\subsection{Shelling the dual}
For this subsection, we restrict to generic hyperbolic slices, that is, the hyperbolic slices where no polynomial has strictly less than $s$ distinct roots. Also note that when $s\leq 1$, hyperbolic posets are simplices (see the proof of Theorem 3.10 in \cite{lien2023hyperbolic}), thus we will only consider the cases when $s\geq 2$.

Recall that $\L_s(F)$ denotes the poset of strata of $\H_s(F)$ partially ordered by inclusion and so we let $\L^\Delta_s(F)$ denote the dual poset. That is, $\L^\Delta_s(F)$ is the set of strata of $H_s(F)$ partially ordered by reverse inclusion. Also, we call the poset $\partial(\L^\Delta_s(F)):=\L^\Delta_s(F)\backslash \emptyset$, the \emph{boundary complex} of $\L^\Delta_s(F)$.

\begin{lemma}\label{3:lem:sc}
    The boundary complex of $\L^\Delta_s(F)$ is a pure simplicial complex of dimension $(n-s-1)$.
\end{lemma}

\begin{proof}
Let $\phi$ be the mapping defined by $\mu \mapsto \{\mu_1,\mu_1+\mu_2,\dots,n\}$ from the poset of compositions of $n$ to the poset of subsets of $[n]$, partially ordered by inclusion. One can easily check that $\phi$ is a poset isomorphism and since the poset of subsets of $[n]$ is a simplex, then so is the poset of compositions.

The poset $\L_s(F)$ can be identified with the poset of compositions that occur in $\H_s(F)$, thus the boundary complex of the dual poset can be thought of as the set $$\{c(H)~|~\ H\in\H_s(F)\},$$ partially ordered by the reverse of our partial order on compositions.

From Proposition \ref{2:prop:pointormaxd}, we know that if a polynomial $H\in \H_s(F)$ has at least $s$ distinct roots, then all the compositions greater than $c(H)$ occur in $\H_s(F)$. Thus, the set of compositions $\{\mu~|~c(H)\leq \mu\}$, is a downwardly closed subposet of the dual poset of compositions. Thus it is a simplex and so $\partial(\L^\Delta_s(F))$ is a simplicial complex. Lastly, from Proposition 3.3 in \cite{lien2023hyperbolic}, we have that $\L^\Delta_s(F)$ is pure and of dimension $n-s-1$.
\end{proof}

\begin{remark}
    The restriction to the generic case in \ref{3:lem:sc} is sufficient, but not necessary. That is, there are examples of non-generic hyperbolic slices where the boundary complex, $\partial(\L^\Delta_s(F))$, is a simplicial complex and examples where it is not. However, the same kind of argument as in Lemma \ref{3:lem:sc} can be used to show that if we remove the empty set and the $0$-dimensional strata from $\L_s(F)$, then the dual poset is a simplicial complex even for non-generic cases.
\end{remark}

We will construct a shelling of $\partial(\L^\Delta_s(F))$ and to do so we shall use a partial order on the zero-dimensional strata of $\H_s(F)$. So let $\gamma_1,\dots,\gamma_k$ be the compositions of length $s$ that occur in $\H_s(F)$, then $F_1:=\H^{\gamma_1}_s(F),\dots,F_k:=\H^{\gamma_k}_s(F)$ are the facets of $\partial(\L^\Delta_s(F))$.

\begin{definition}\label{3:def:shell}
Let "$\leq_p $" denote the partial order on $F_1,\dots,F_k$ that is generated by the covering relations $\{H\}=F_i<_p F_j=\{G\}$ if there is a one-dimensional stratum $R$ of $\H_s(F)$ for which $H$ is minimal and $G$ is maximal.
\end{definition}

\begin{lemma}\label{3:lem:link}
Let $S$ be a stratum of $\H_s(F)$. If $H\in F_i$ is the minimal (resp. maximal) polynomial of the stratum $S$ and $F_j \subseteq S$, then $F_i\leq_p F_j$ (resp. $F_i\geq_p F_j$).
\end{lemma}

\begin{proof}
Since $H$ is minimal in $S$, then either $F_i=F_j$ or there is a one-dimensional stratum, $R_1\subseteq S$, for which $G\in F_j$ is maximal. Otherwise $G$ would be minimal in $S$ by Lemma \ref{2:lem:maxiffmax}. By Theorem \ref{2:thm:minmax}, the stratum $R_1$ also contains a minimal polynomial $Q\in F_m$ for some $m$ and therefore $F_m<_p F_j$.

And by the same argument as above, either $Q=H$ or there must be a one-dimensional stratum $R_2\subseteq S$, for which $Q$ is maximal. We see that by continuing this process we must eventually end up at $H$ and so $F_i\leq_p F_j$. The argument for maximal polynomials is analogous.
\end{proof}

\begin{definition}
Let $\leq$ and $\leq^*$ be partial orders on a set $P$. Then we say $\leq$ is \textbf{finer} than $\leq^*$ if $a\leq^* b$, for some $a,b\in P$, implies $a\leq b$.
\end{definition}

\begin{theorem}\label{3:thm:shell}
    Let $\leq$ be a total order on $\{F_1,\dots,F_k\}$ that is finer than $\leq_p$, then the total order (and its reverse order) induces a shelling of $\partial(\L_s^\Delta(F))$.
\end{theorem}

\begin{proof}
We can assume by relabelling that $F_1< \dots < F_k$. As we are shelling the boundary complex of the dual poset we will first rephrase Definition \ref{3:def:shell} to suit our setting: 

$F_1,\dots,F_k$ is a shelling of $\partial(\L_s^\Delta(F))$ if for any $i\in \{2,\dots,k\}$ and any $j\in [i-1]$, there is an $r\in [i-1]$ such that the minimal stratum containing both $F_i$ and $F_j$ also contains a one-dimensional stratum, $R$, which contains both $F_i$ and $F_r$. Note that this guarantees that in the dual poset, the intersection of the facets $F_i$ and $F_j$ is contained in the ridge $R$, which again is contained in the facets $F_i$ and $F_r$.

By Lemma \ref{3:lem:sc}, the boundary complex of the dual poset is simplicial, thus there is a smallest stratum, $S$, containing both $F_i$ and $F_j$. The polynomial $H\in F_i$ cannot be the minimal polynomial of $S$, otherwise $F_j<_p F_i$ by Lemma \ref{3:lem:link}, which would contradict $\leq$ being finer than $\leq_p$. 

So by Lemma \ref{2:lem:maxiffmax}, $H$ is maximal for a one-dimensional stratum $R\subset S$. Let $G\in F_r$ be the minimal polynomial of $R$. then $F_r<_p F_i$ by Lemma \ref{3:lem:link} and therefore $F_r<F_i$ since $\leq$ refines $\leq_p$ and so $r\in[i-1]$.
\end{proof}

\begin{corollary}\label{3.cor:sph}
    The boundary complex $\partial(\L_s^\Delta(F))$ is a combinatorial $(n-s-1)$-sphere.
\end{corollary}

\begin{proof}
    Any ridge of $\partial(\L_s^\Delta(F))$ corresponds to an edge $\H_s^\mu(F)\in\L_s(F)$. By Lemma \ref{2:lem:compcontr}, $\H_s^\mu(F)$ is compact and thus has two endpoints. By Proposition \ref{2:prop:pointormaxd}, those endpoints are polynomials with $s$ distinct roots and they have distinct compositions. Thus there are exactly two vertices in $\H_s^\mu(F)$, that is, any ridge in $\partial(\L_s^\Delta(F))$ is contained in exactly two facets. So from Proposition \ref{prelim:plsphere}, $\partial(\L_s^\Delta(F))$ is a combinatorial $(n-s-1)$-sphere.
\end{proof}

\subsection{UBT and g-theorem}

Due to Corollary \ref{3.cor:sph}, we can make use of some previously established results for simplicial spheres to say something about the number of $i$-dimensional strata in $L_s(F)$.

\begin{definition}
    Let $d=\dim(\H_s(F))$ and for $i\in \{0,1,\dots,d\}$, let $f_i$ denote the number of $i$-dimensional strata of $\H_s(F)$. Then $(f_0,\dots,f_d)$ is the \emph{f-vector} of $\L_s(F)$. 
\end{definition}

As we are looking at the dual poset of $\L_s(F)$, note that generically $f_i$ is the number of $(d-i-1)$-dimensional simplices in $\L_s^\Delta(F)$ (we consider the empty set to have dimension $-1$). Thus $(f_d,\dots,f_0)$ is the f-vector of the simplicial complex $\partial(\L_s^\Delta(F))$. Although the f-vector has an easy interpretation, it is often more convenient to work with the \emph{h-vector}, $(h_0,..,h_d)$, of $\partial(\L_s^\Delta(F))$, where $$h_i=\sum_{j=0}^i(-1)^{i-j}\binom{d-j}{i-j}f_{d-j}.$$

Note that when $\H_s(F)$ is generic, then this definition is the same as the usual definition for simplicial complexes (see Definition 8.18 in \cite{ziegler2012lectures}) since the simplicial complex $\partial(\L_s^\Delta(F))$ has the f-vector $(f_d,\dots,f_0)$. We can pass from the h-vector to the f-vector by using the following relations (see page 249 of \cite{ziegler2012lectures}): $$f_{d-i}=\sum_{j=0}^i\binom{d-j}{i-j}h_j.$$ In our setting the h-vector has the following interpretation:

\begin{corollary}\label{3:cor:palind}
Let $(h_0,\dots,h_d)$ be the $h$-vector of $\partial(\L_s^\Delta(F))$. Then $h_i$ is the number of polynomials in $\H_s(F)$ that are maximal for exactly $i$ one-dimensional strata. Similarly, $h_i$ is also the number of polynomials in $\H_s(F)$ that are minimal for exactly $i$ one-dimensional strata.
\end{corollary}

\begin{proof}
Let again $\leq$ be a total order on $\{F_1,\dots,F_k\}$ that is finer than $\leq_p$ and assume that $F_1< \dots <F_k$, then by Theorem \ref{3:thm:shell}, $F_1,\dots,F_k$ is a shelling of $\partial(\L_s^\Delta(F))$. We denote by $V_j$ the set of vertices of $F_j$ and by $R_j\subseteq V_j$ the \emph{restriction} of $F_j$, which is defined as the subset of vertices of $F_j$, such that for every $v\in R_j$ the set $V_j \setminus \{v\}$ lies in $F_m$ for some $m<j$. Then from the first part of section 8.3 in \cite{ziegler2012lectures} we have that $h_i$ is equal to
\[|\{j:|R_j|=i\}|.\]
Let $v\in \R_j$ and let $m<j$, such that $V_j\setminus\{v\}\subset F_m$. Then $F_m$ and $F_j$ are joined by a one-dimensional stratum $E$ of $\H_s(F)$ and since $F_m< F_j$, then $H\in F_j$ is maximal in $E$. Conversely, for any one-dimensional stratum $E'$ of $\H_s(F)$ such that $H_j\in F_j$ is maximal and $H_r\in F_r$ is minimal in $E'$, we have that $F_r<F_j$ and $V_j\backslash \{v\}\subset F_r$ for some $v\in V_j$.

Thus $|R_j|$ counts the number of one-dimensional strata of $\H_s(F)$ for which $H\in F_j$ is maximal. And so $h_i$ counts the number of zero-dimensional strata that are maximal for exactly $i$ one-dimensional strata. If we now take the reverse order (which by Theorem \ref{3:thm:shell} is also a shelling), then with an analogous argument we find that $h_i$ is equal to the number of vertices that are minimal for exactly $i$ one-dimensional strata.
\end{proof}

If a polynomial is maximal for $i$ one-dimensional strata, it must be minimal for the other $n-s-i$ one-dimensional strata that contain it. Thus Corollary \ref{3:cor:palind} implies that the $h$-vector of $\partial(\L_s^\Delta(F))$ must be palindromic. That is, it satisfies the \emph{Dehn-Sommerville equations}: 
\[h_i=h_{d-i} \ \text{ for all } \ i \in \lfloor d/2\rfloor.\]

Moreover, since $\partial(\L_s^\Delta(H))$ is a combinatorial sphere, we can obtain further properties of its $h$-vector from the $g$-conjecture for simplicial spheres that was recently proven in \cite{adiprasito2018combinatorial}. In order to state those results, we have to introduce some notation.

Firstly, for $k,i\in \N$ there are unique integers $a_i\geq \dots\geq a_1\geq 0$ such that
\[\binom{a_i}{i}+\binom{a_{i-1}}{i-1}+\dots+\binom{a_1}{1} \text{ (see page 265 in \cite{ziegler2012lectures})}.\]
\begin{definition}
We say that $g=(g_0,\dots,g_r)\in \mathbb{N}_0^r$ is a \emph{Macaulay (or $M$-) vector}, if $g_0=1$ and for any $i\in [r-i]$
\[g_{i+1}\leq \binom{a_i+1}{i+1}+\binom{a_{i-1}+1}{i}+\dots+\binom{a_1+1}{1+1},\]
where
\[g_i=\binom{a_i}{i}+\binom{a_{i-1}}{i-1}+\dots+\binom{a_1}{1}\] 
is the unique representation of $g_i$ introduced above.
\end{definition}

\begin{corollary}["g-theorem"]\label{3:cor:g-thm}
    Let $\H_s(F)$ be generic, then the $h$-vector $(h_0,\dots,h_{n-s})$ of $\partial(\L_s^\Delta(F))$ satisfies
    \begin{enumerate}
        \item $h_i = h_{n-s-i} \text{ for all } i\leq \floor{(n-s)/2}$ (Dehn-Sommerville),
        \item $h_i\geq h_{i-1} \text{ for all } i\leq\floor{(n-s)/2}$ (lower bound) and
        \item $(h_0,h_1-h_0,\dots,h_{\floor{(n-s)/2}}-h_{\floor{(n-s)/2}-1})$ is a Macaulay vector.
    \end{enumerate}
\end{corollary}

Since we have situations where $\L^\Delta_s(F)$ is isomorphic to non-simplicial polytopes where the g-theorem does not hold, we cannot extend the theorem in its entirety to the general setting. See for instance Example \ref{1:ex:pyramid}, where the h-vector is not palindromic. However, the third condition in Corollary \ref{3:cor:g-thm} can be used to deduce the Upper Bound Theorem for polytopes (see Section 3 in \cite{mcmullen1971numbers}) and this is a bound that we can extend to the general case.

To extend the generic bound, we show that the component-wise maximal $f$-vector of hyperbolic posets is attained in some generic case. In the following we identify $\L_s(F)$ with the poset of compositions that occur in $\H_s(F)$.

\begin{proposition}\label{3:prop:subdiv}
    Suppose $F$ has no repeated roots and $n-s>0$, then there is a $\delta>0$ such that for all $\epsilon$ with $0<\epsilon<\delta$, 
    \begin{enumerate}
        \item $\H_s(F+\epsilon T^{n-s})$ is generic,
        \item $\lambda\in \L_s(F+\epsilon T^{n-s}) \implies \lambda\geq \mu \text{ for some }\mu\in \L_s(F)$,
        \item $\mu\in\L_s(F) \ \& \ \ell(\mu)\geq s \implies \mu \in\L_s(F+\epsilon T^{n-s})$ and
        \item for any $\mu\in \L_s(F)$ with $\ell(\mu)<s$, there is a $\lambda\in \L_s(F+\epsilon T^{n-s})$ of length $s$ such that $\lambda\geq \mu$ and $\lambda$ is incomparable with all other compositions of length at most $s$ in $\L_s(F)$.
    \end{enumerate}
\end{proposition}

\begin{proof}
    By Proposition \ref{2:prop:pointormaxd}, $\H_s(F)$ is of dimension $n-s>0$ and $\H_{s-1}(F)$ is of dimension $n-s+1$. Since $F$ is in the interior of $\H_{s-1}(F)$, we can choose a $\delta>0$ such that $B_\delta(F)\subset \H_s(F)$. Since there are finitely many polynomials in $\H_{s-1}(F)$ with at most $s-1$ distinct roots we can choose a $\delta$ such that for all $\epsilon$ with $0<\epsilon<\delta$, $\H_s(F+\epsilon T^{n-s})$ contains only polynomials with at least $s$ distinct roots.

    For the second statement, let $\lambda\in \L_s(F+\epsilon T^{n-s})$ and let $H$ be the minimal polynomial of $\H_{s-1}^\lambda(F)$. By Theorem \ref{2:thm:minmax}, $H$ has at most $s-1$ distinct roots. Thus we either have $H\in\H_s(F)$ and $c(H)\leq \lambda$ or $H\not\in\H_s(F)$ and $\lambda\in\L_s(F)$.

    For the third statement, let $Q$ be a polynomial in $\H_s(F)$ with at least $s$ distinct roots and composition $\mu$. By Proposition \ref{2:prop:pointormaxd}, $\H_{s-1}^\mu(F)$ is of dimension $\ell(\mu)-s+1>0$. By Theorem \ref{2:thm:minmax}, $\H_{s-1}^\mu(F)$ has a maximal polynomial, $G$, with at most $s-1$ distinct roots. Thus the $((s-1)+1)^{th}$ coefficient of $G$ is at least as large as the $((s-1)+1)^{th}$ coefficient of $F$ plus $\delta$. Since $\H_{s-1}^\mu(F)$ is contractible the intersection of $\H_{s-1}^\mu(F)$ and $\H_s(F+\epsilon T^{n-s})$ is nonempty. So $\H_s^\mu(F+\epsilon T^{n-s})$ is nonempty and contains no polynomial with strictly less than $s$ distinct roots. Thus, by Proposition \ref{2:prop:pointormaxd}, $\H_s^\mu(F+\epsilon T^{n-s})$ contains a polynomial with composition $\mu$.

    For the last statement, let $P$ be a polynomial with at most $s-1$ distinct roots and composition $\mu$. Since $P$ is neither the minimal nor the maximal polynomial of $\H_{s-1}(F)$, then $s-1>1$ by the  theorem in \cite{meguerditchian1992theorem} and so by Lemma \ref{2:lem:maxiffmax}, there is a one-dimensional stratum $\H_{s-1}^\lambda(F)$ for which $P$ is the minimal polynomial. Similar to the argument above, $\H_s^\lambda(F+\epsilon T^{n-s})$ must therefore contain a polynomial with composition $\lambda$. Also, by Proposition \ref{2:prop:pointormaxd}, $\ell(\lambda)=s$ since $\H_s^\lambda(F+\epsilon T^{n-s})$ is generic and zero-dimensional. Lastly, by Theorem \ref{2:thm:minmax}, $P$ is the unique minimal polynomial of $\H_{s-1}^\lambda(F)$, thus $c(P)$ is the only composition in $\L_s(F)$ that is smaller than or equal to $\lambda$.
\end{proof}

\begin{remark}\label{3:rem:subd}
    We see in Proposition \ref{3:prop:subdiv} that a non-generic $\H_s(F)$ can be obtained from some generic $\H_s(H)$ by "contracting" some of the strata of $\H_s(H)$ to points. This corresponds to merging some of the faces of $\partial(\L_s^\Delta(H))$. In other words if $\partial(\L_s^\Delta(F))$ is a polytopal complex, then the simplicial complex $\partial(\L_s^\Delta(H))$ is a simplicial subdivision of $\partial(\L_s^\Delta(F))$. Thus whenever $\partial(\L_s^\Delta(F))$ is a polytopal complex it is also a combinatorial sphere. However, we do not know if $\partial(\L_s^\Delta(F))$ is a polytopal complex in general and thus we have restricted ourselves to the generic case.
\end{remark}

Due to the preceding remark, we have the following weaker conjecture than Conjecture \ref{2:conj:polyt}.

\begin{conjecture}\label{3:conj:comb.sphere}
The boundary complex $\partial(\L^\Delta_s(F))$ is a polytope complex and thus by Remark \ref{3:rem:subd}, a combinatorial sphere. 
\end{conjecture}

\newpage
To state the bound for the general case we need another definition.

\begin{definition}
    We define
    \[\abb{\phi_d}{\R}{\R^d}{x}{(x,x^2,\dots,x^d)}\]
    to be the $d$-th \emph{moment curve}. If $x_1,\dots,x_m\in \R$ are distinct, we say that the convex hull of $\phi_d(x_1),\dots,\phi_d(x_m)$ is the $d$-dimensional \textbf{cyclic polytope} on $m$ vertices.
\end{definition}

\begin{corollary}[Upper Bound Theorem]\label{3:cor:ubt}
    Let $(f_0,\dots,f_{n-s})$ be the $f$-vector of $\L_s(F)$. If $c_{i}$ is the number of $i$-dimensional faces of the $(n-s)$-dimensional cyclic polytope with $f_{n-s-1}$ vertices then 
    $$f_{n-s-i}\leq c_{i-1} \ \forall \ i\in [n-s].$$
\end{corollary}

\begin{proof}
    By Proposition \ref{2:prop:pointormaxd}, we may assume $H_s(F)$ is $(n-s)$-dimensional where $n-s>0$ and we may assume $F$ has no repeated roots. Then, by Proposition \ref{3:prop:subdiv}, there is an $\epsilon>0$ such that $\H_s(F+\epsilon T^{d-s})$ is generic and whose f-vector is component-wise an upper bound on the f-vector of $\H_s(F)$. Thus we can reduce to the case when $\H_s(F)$ is generic.
    
    When $\H_s(F)$ is generic we know that the h-vector of $\partial(\L_s^\Delta(F))$ is palindromic. From this, it can be shown that the upper bound on the f-vector is obtained by establishing the following upper bound on the h-vector (see chapter 8.4 in \cite{ziegler2012lectures}):
    $$h_i\leq \binom{f_{n-s-1}-n+s-1+i}{i}.$$

    The claim now follows directly from the Upper Bound Theorem for simplicial spheres (Cor. 5.3 in \cite{stanley1975upper}) since $\partial(\L_s^\Delta(F))$ is a combinatorial sphere for generic $\H_s(F)$ by Corollary \ref{3.cor:sph}.
    \end{proof}

\begin{remark}\label{3:rem:bound}
    In \cite{riener2012degree} (Theorem 4.2) it was shown that the extremal points of the convex hull of $\H_s(F)$ are contained in the subset of polynomials of $\H_s(F)$ with at most $s$ distinct roots. And since 
    Corollary \ref{3:cor:ubt} together with Exercise 0.9 in \cite{ziegler2012lectures} gives us an explicit upper bound on the number of polynomials in $\H_s(F)$ with at most $s$ distinct roots, it also gives us an upper bound on the number of local extremal points. This improves the bound given in Theorem 2.14 and Remark 2.15 in \cite{riener2024linear} to the following
    \begin{align*}
    f_0&\leq \begin{cases}
    \binom{n-1-(n-s)/2}{(n-s)/2}+\binom{n-2-(n-s)/2}{(n-s)/2-1}, & \text{if $n-s$ is even}\\
    2\binom{n-2-(n-s-1)/2}{(n-s-1)/2}, & \text{if $n-s$ is odd}
    \end{cases}\\
    &= \begin{cases}
    \binom{(n+s)/2-1}{s-1}+\binom{(n+s)/2-2}{s-1}, & \hspace{8mm} \text{if $n-s$ is even}\\
    2\binom{(n+s-3)/2}{s-1}, & \hspace{8mm} \text{if $n-s$ is odd}
    \end{cases}.
    \end{align*}
\end{remark}

We have computationally verified that the bound in Remark \ref{3:rem:bound} can be attained when $n\leq 8$ and $s\leq n$ and one can also use Proposition \ref{2:prop:pointormaxd} to argue that the bound is attained when $s\leq 2$ and when $s\geq n-1$. Therefore we have the following conjecture:

\begin{conjecture}\label{3:conj:bound}
The bound stated in Remark \ref{3:rem:bound} is sharp.
\end{conjecture}

\section{Improving Timofte's Degree principle}\label{sec:timofte}

Throughout the section, we denote by 
$\C(n,s)$ and by $\P(n,s)$ the set of all compositions and partitions, respectively, of $n$ into $s$ parts and by $\C_{\min}(n,s)$ and $\P_{\min}(n,s)$ the compositions and partitions that correspond to a minimal polynomial in some generic hyperbolic slice. 

Timofte showed in \cite{timofte2003positivity} the so-called "degree principle": Symmetric polynomials of degree at most $s$ have a common real root if and only if they have a common real root with at most $s$ distinct coordinates. We want to improve this result by considering subsets of the set of points with at most $s$ distinct coordinates. To this end, we introduce some notation:

\begin{definition}
Let $P\subseteq \P(n,s)$. We say that $P$ is a \emph{$(n,s)$-Vandermonde covering}, if for every hyperbolic slice $\H_s(F)$ there is a partition $q\in P$ and a polynomial $G\in \H_s(f)$ with corresponding partition $p(G)$ such that $q\geq p(G)$.
\end{definition}

Since we are interested in symmetric polynomials, the roots of the polynomials are closed under permutations. So we identify the orbit types of points in $R^n$ by partitions. Instead of considering all points with at most $s$ distinct coordinates in the degree principle, we want to consider only points with orbit types corresponding to a partition in a Vandermonde covering.

\begin{definition}
Let $P\subseteq \P(n,s)$. We denote by
\[A_P:=\left\{ (\underbrace{x_1,\dots,x_1}_{q_1-\text{times}},\underbrace{x_2,\dots,x_2}_{q_2-\text{times}},\dots,\underbrace{x_s,\dots,x_s}_{q_s-\text{times}})\in\R^n ~\middle|~ q\in P \right\} \]
the set of points with coordinate multiplicities corresponding to a partition in $P$.
\end{definition}

The following theorem motivates the name "Vandermonde covering" and can also be seen as a strenghtening of the degree principle presented in \cite{riener2012degree}.

\begin{theorem}\label{4:theorem:char_timofte_covering}
Let $P\subseteq \P(n,s)$. The following are equivalent:
\begin{enumerate}
    \item $P\subseteq \P(n,s)$ is a $(n,s)$-Vandermonde covering.
    \item For all $k\in \N$ and all symmetric polynomials $F_1,\dots,F_k\in \R[\X]$ of degree at most $s$
    \[V_\R(F_1,\dots,F_k)\neq \emptyset \Leftrightarrow V_\R(F_1,\dots,F_k)\cap A_P\neq\emptyset. \]
    \item For all $a\in \R^s$, the Vandermonde variety
    \[\V(a)\neq \emptyset \Leftrightarrow \V(a)\cap A_P\neq\emptyset. \]
\end{enumerate}
\end{theorem}

\begin{proof}
(1)$\Rightarrow$(2): 
    Let $P\subseteq \P(n,s)$ be a $(n,s)$-Vandermonde covering and let $x\in V_\R(F_1,\dots,F_k)$. Consider
    \[ F:= T^n - E_1(x) T^{n-1} + \dots + (-1)^n E_n(x) \]
    with roots $x_1,\dots,x_n$.
    Then there is a partition $q\in P$ and a polynomial $g\in \H_s(F)$ with corresponding partition $p(g)\leq q$ and roots 
    \[\tilde x=(\tilde x_1,\dots,\tilde x_n)\in A_P,\]
    because $P$ is a $(n,s)$-Vandermonde covering. Since $F_1,\dots,F_k$ are polynomials of degree at most $s$, we can write
    \[F_1=G_1(E_1,\dots,E_s),\dots,F_k=G_k(E_1,\dots,E_s)\]
    for some $G_1,\dots,G_k\in \R[Y_1,\dots,Y_s]$ by Lemma \ref{4:lem:cordian}. Now
    \[0=F_i(x)=G_i(E_1(x),\dots,E_s(x))=G_i(E_1(\tilde x),\dots,E_s(\tilde x))=F_i(\tilde x)\]
    and therefore $\tilde x \in V_\R(F_1,\dots,F_k)$.
\\(2)$\Rightarrow$(3): This is clear, because $E_i-a_i$ is symmetric of degree $i$.
\\(3)$\Rightarrow$(1):
    Assume (3) holds. Let $F=T^n - c_1 T^{n-1} + \dots + (-1)^n c_n$ be a hyperbolic polynomial with roots $x=(x_1,\dots,x_n)\in \R^n$. Then the Vandermonde variety $\V(-c_1,\dots,(-1)^sc_s)$ contains $x$ by construction and is therefore nonempty. By (3) there is an $\tilde x\in \V(-c_1,\dots,(-1)^sc_s)\cap A_P$, i.e.
    \[E_1(\tilde x)=c_1,\dots, E_s(\tilde x)=c_s.\]
    Now
    \[G:= T^n - E_1(\tilde x) T^{n-1} +\dots+ (-1)^n E_n(\tilde x)\]
    is a polynomial in $\H_s(F)$ with corresponding partition $p(G)\leq q$ for some $q\in P$.
\end{proof}

In the light of Theorem \ref{4:theorem:char_timofte_covering}, the degree principle can be interpreted as thefact that $\P(n,s)$ is a Vandermonde covering which follows from for example Theorem \ref{2:thm:minmax}.

\subsection{General bounds on Vandermonde coverings}

Since every generic hyperbolic slice has a unique minimal polynomial with a corresponding alternate odd composition, we get the following Vandermonde covering:

\begin{theorem}\label{4:upper_bound}
The set $\P_{\min}(n,s)$ is a $(n,s)$-Vandermonde covering of size $\left|\P\left(n - \ceil*{\frac{s}{2}},\floor*{\frac{s}{2}}\right)\right|$.
\end{theorem}
\begin{proof}
Follows directly from Theorem \ref{2:thm:minmax} or from the less general version presented in \cite{meguerditchian1992theorem}.
\end{proof}

We show below that $\P_{\min}(n,s)$ is in general not the smallest Vandermonde covering. In order to estimate how good this Vandermonde covering is, we want to get lower bounds on the size of Vandermonde coverings. To this end, we need the following definition and some properties of the set of minimal and maximal partitions.

\begin{definition}
    We denote by $\P(n)$ the set of all partitions of $n$. The partial order on the set of all compositions of $n$ induces a partial order $\leq$ on $\P(n)$: For $p,q\in\P(n)$ we write $p\leq q$ if $p$ can be obtained from $q$ by summing some of the parts in $q$ and then reordering.
    Additionaly, if $\ell(q)=\ell(p)+1$, then we say $q$ \emph{covers} $p$. 
\end{definition}
Note that for two partitions $p$ and $q$, $p\leq q$ if and only if there are permutations $\sigma$ and $\tau$, such that $\sigma p \leq \tau q$ as compositions.

\begin{lemma}\label{4:lem_lower_Lemma}
\begin{enumerate}
    \item $\P_{\min}(n,s-1)\subseteq\P_{\max}(n,s-1)$.
    \item $|\P_{\min}(n,s)|=|\P_{\max}(n-1,s-1)|$.
    \item Let $P\subseteq \P(n,s)$ be a $(n,s)$-Vandermonde covering, then $P$ has to cover $\P_{\max}(n,s-1)$.
    \item Every partition in $\P(n,s)$ covers at most 
    \[\frac{\ceil*{\frac{s-1}{2}}^2+\ceil*{\frac{s-1}{2}}}{2}
    =\frac{\ceil*{\frac{s-1}{2}\rceil \lceil\frac{s+1}{2}}}{2}\] 
    partitions in $\P_ {\max}(n,s-1)$.
\end{enumerate}
\end{lemma}
\begin{proof}
\begin{enumerate}
    \item Let $p\in \P_{\min}(n,s-1)$. Then $p$ is of the form
    \[p=\big(p_1,\dots,p_{\floor*{\frac{s-1}{2}}},\underbrace{1,\dots,1}_{\ceil*{\frac{s-1}{2}}-\text{times}}\big)\]
    since it corresponds to an alternate odd composition $(\mu_1,\dots,\mu_{s-1})$ by Theorem \ref{2:thm:minmax}. Now $p$ corresponds also to the alternate even composition $(\mu_{s-1},\mu_1,\mu_2,\dots,\mu_{s-2})$ 
    %defined by
    %\[\mu_{s-i}:=\begin{cases}
    %p_{\frac{i+1}{2}} & \text{if $i$ is odd and $i\leq s-2$}\\
    %    1  & \text{else}
    %\end{cases} \]
    and therefore $p\in \P_{\max}(n,s-1)$.
    \item Follows directly from the bijection
    \[\abb{\phi}{\C_{\max}(n-1,s-1)}{\C_{\min}(n,s)}{(\mu_1,\dots,\mu_{s-1})}{(\mu_1,\dots,\mu_{s-1},1)}.\]
    \item Let $\mu\in \C_{\max}(n,s-1)$ and let $F$ be a polynomial with root multiplicities corresponding to $\mu$. Then by Theorem \ref{2:thm:minmax} and Proposition \ref{2:prop:pointormaxd}, $F$ is the maximal polynomial of $\H_{s-1}^\mu(F)$ and $\H_{s-1}^\mu(F)$ is of dimension $n-s+1$. Now for $\epsilon>0$ small enough there is some monic polynomial $H$ of degree $n-s$, such that $\H_{s}(F-\epsilon H)$ is $(n-s)$-dimensional with zero-dimensional strata corresponding to all compositions that cover $\mu$ by Proposition \ref{2:prop:pointormaxd}. Since $P$ is a $(n,s)$-Vandermonde covering, there has to be a $q\in P$ such that $q\geq p(G)$ for some $G\in \H_{s}(F-\epsilon T^s)$ and so we have $q\geq p(G)>p(F)$.
    \item In order for $p\in \P(n,s)$ to cover a partition in $\P_{\max}(n,s-1)$ there can be at most $\floor*{\frac{s-1}{2}}+1$ entries different from $1$ in $p$. One can now obtain all partitions in $\P_{\max}(n,s-1)$ that are covered by $p$ by summing two of the first $\floor*{\frac{s-1}{2}}+1$ entries in $p$. So $p$ covers at most
    \[\binom{\floor*{\frac{s-1}{2}}+1}{2}=\frac{\ceil*{\frac{s-1}{2}}^2+\ceil*{\frac{s-1}{2}}}{2}
    =\frac{\ceil*{\frac{s-1}{2}\rceil \lceil\frac{s+1}{2}}}{2}\]
    partitions in $\P_{\max}(n,s-1)$.
\end{enumerate}
\end{proof}

From this lemma, we get the following lower bounds on the size of any Vandermonde covering:

\begin{proposition}\label{4:trivial_lower_bound}
Let $P\subseteq \P(n,s)$ be a $(n,s)$-Vandermonde covering, then \[|P|\geq \ceil*{\frac{2\left|\P\left(n - \ceil*{\frac{s}{2}},\floor*{\frac{s}{2}}\right)\right|}{\ceil*{\frac{s-1}{2}} \ceil*{\frac{s+1}{2}}}}.\]
\end{proposition}
\begin{proof}
By Lemma \ref{4:lem_lower_Lemma} (3), $P$ has to cover $\P_{\max}(n,s-1)$. Every partition in $\P(n,s)$ covers at most 
\[\frac{\ceil*{\frac{s-1}{2}\rceil \lceil\frac{s+1}{2}}}{2}\] 
partitions in $\P_ {\max}(n,s-1)$ by Lemma \ref{4:lem_lower_Lemma} (4). From the pigeonhole principle, we get that we need at least
\[\ceil*{\frac{2|\P_{\max}(n,s-1)|}{\ceil*{\frac{s-1}{2}} \ceil*{\frac{s+1}{2}}}}
=\ceil*{\frac{2|\P_{\min}(n,s)|}{\ceil*{\frac{s-1}{2}} \ceil*{\frac{s+1}{2}}}}=\ceil*{\frac{2\left|\P\left(n - \ceil*{\frac{s}{2}},\floor*{\frac{s}{2}}\right)\right|}{\ceil*{\frac{s-1}{2}} \ceil*{\frac{s+1}{2}}}}\]
partitions to have at least one partition from every generic slice.
\end{proof}

This lower bound can be improved by considering recursively those maximal partitions that have $i$ entries different from $1$, which is the main idea behind the following theorem.

\begin{theorem}\label{4:recursive_lower_bound}
Let $P\subseteq \P(n,s)$ be a $(n,s)$-Vandermonde covering. Then \[|P|\geq \sum_{i=0}^{\floor*{\frac{s}{2}}} B_i, \]
where $B_0:=0$, $B_1:=1$ and
\[ B_{i}:= \ceil*{
2\frac{|\P(n-s+1,i)|-i B_{i-1}-B_{i-2}}{i^2+i}
} \]
for all $i\in \left\{2,\dots,\floor*{\frac{s}{2}}\right\}$.
\end{theorem}
\begin{proof}
Denote by
\[P_i:=\big\{ q\in \P_{\max}(n,s-1) ~ \big| ~ 
|\{j\in [n] ~|~ q_j\neq 1\}|=i \big\} \] 
the partitions in $\P_{\max}(n,s-1)$ that have exactly $i$ entries different from $1$. Note the following:
\begin{enumerate}
    \item $|P_i|=|\P(n-s+1,i)|.$
    \item Every partition in $\P(n,s)$ covers at most $\binom{i+1}{2}=\frac{i^2+i}{2}$ partitions in $P_i$ by a similar argument as in the proof of Lemma \ref{4:lem_lower_Lemma} (4).
    \item A partition in $\P(n,s)$ that covers a partition in $P_i$, covers at most $i+1$ partitions in $P_{i+1}$ and at most one partition in $P_{i+2}$.
\end{enumerate}
Now, in order to cover all partitions in $\P_{\max}(n,s-1)$, we have to cover all partitions in $P_i$ for all $i\in \left[\floor*{\frac{s}{2}}\right]$. Combining (1), (2) and (3) we get recursively:
We need $B_1=1$ partition in $\P(n,s)$ to cover $P_1$. It covers at most $(1+1)B_1$ partitions in $P_2$ and at most $B_1$ partitions in $P_3$. To cover the at least $P_2-2 B_1$ remaining many partitions in $P_2$ we need by the pigeonhole principle at least 
\[B_2 = \ceil*{
\frac{|P_2|-2 B_{1}-B_{0}}{(2^2+2)/2}
} = \ceil*{
2\frac{|\P(n-s+1,2)|-2 B_{1}-B_{0}}{2^2+2}
} \]
additional partitions in $\P(n,s)$. Those partitions cover again at most $(2+1)B_2$ partitions in $P_3$ and at most $B_2$ partitions in $P_4$. To cover at least the $P_3-3 B_2-B_1$ remaining partitions in $P_3$ we need by the pigeonhole principle at least
\[B_3= \ceil*{
\frac{|P_3|-3 B_{2}-B_{1}}{(3^2+3)/2}
} = \ceil*{
2\frac{|\P(n-s+1,3)|-3 B_{2}-B_{1}}{3^2+3}
}\]
additional partitions in $\P(n,s)$. In general, if $B_i$ denotes the number of additional partitions needed to cover the remaining partitions in $P_i$, then
\[ B_{i}:= \ceil*{
2\frac{|\P(n-s+1,i)|-i B_{i-1}-B_{i-2}}{i^2+i}
}. \]
In total, we need at least $ \sum_{i=0}^{\floor*{\frac{s}{2}}} B_i $ partitions in $\P(n,s)$ to cover all partitions in $\P_{\max}(n,s-1)$.
\end{proof}

\subsection{Algorithmic improvements of Vandermonde coverings.}

In the following we want to present an algorithmic approach on how to obtain smaller - possibly optimal - Vandermonde coverings for small $s$ and $n$. To this end, we try to characterize if a set of compositions $S\subset \C(n,s)$ corresponds to the set of zero-dimensional strata of some hyperbolic slice. 

\begin{definition}
Let $S\subseteq \C(n,s)$. We call the upward closure of $S$
\[\L(S) := \{ \lambda ~|~ \text{there is a $\mu\in S$ with $\mu\leq \lambda$}\}\cup {(n)}\]
the \emph{poset} of $S$. We say that $\L(S)$ is a \emph{potential hyperbolic poset}, if for every $\lambda \in \L(S)$ there are unique $\mu_{\min},\mu_{\max}\in S$, such that
\begin{enumerate}
    \item $\mu_{\min}/\lambda$ is alternate odd and
    \item $\mu_{\max}/\lambda$ is alternate even.
\end{enumerate}
 Furthermore, we say that $\L(S)$ is a \emph{realizable hyperbolic poset}, if it is isomorphic to a hyperbolic poset $\L_s(F)$.
\end{definition}

\begin{remark}
One can also consider more general potential hyperbolic posets, where $S$ is a set of compositions of $n$ into at most $s$ parts. For this we construct $\L(S)$ analagous to Algorithm 3.12 in \cite{lien2023hyperbolic}, that is, by first taking the join of pairwise distinct elements of $S$ and then the upward closure of these joins.
\end{remark}

Theorem \ref{2:thm:minmax} states that every realizable hyperbolic poset is a potential hyperbolic poset. Note that the boundary complex of the dual $\partial(\L^\Delta(S))$ of a potential hyperbolic poset $\L(S)$ is a simplicial complex. One can see that the arguments in the proof of shellability in Section 2 only uses the structure of a poset, Theorem \ref{2:thm:minmax} and the fact that the the boundary complex of the dual is a pure simplicial complex, so all combinatorial results also hold for potential hyperbolic posets. In particular, we get the following:

\begin{theorem}\label{4:shellable_sphere_g-theorem}
Let $\L(S)$ be a potential hyperbolic poset and denote by $\partial(\L^\Delta(S))$ the boundary complex of the dual poset of $\L(S)$. Then
\begin{enumerate}
    \item $\partial(\L^\Delta(S))$ is a shellable simplicial complex and therefore a combinatorial sphere.
    \item The $h$-vector of $\partial(\L^\Delta(S))$ satisfies the "g-theorem", i.e. the inequalities stated in Corollary \ref{3:cor:g-thm}. 
    \item $|S|\leq \begin{cases}\binom{(n+s)/2-1}{s-1}+\binom{(n+s)/2-2}{s-1}, & \hspace{8mm} \text{if $n-s$ is even}\\
    2\binom{(n+s-3)/2}{s-1}, & \hspace{8mm} \text{if $n-s$ is odd}
    \end{cases}.$
\end{enumerate}
\end{theorem}

Since all the known combinatorial properties of the poset of a generic hyperbolic slice hold for all potential hyperbolic posets, we don't know any combinatorial way to distinguish potential from realizable hyperbolic posets. Moreover, by computationally realizing all hyperbolic posets up to $s\leq n\leq 6$, we state the following conjecture:

\begin{conjecture}\label{4:conj_potential_realizable}
Every potential hyperbolic poset is realizable.
\end{conjecture}

Since it is easy to check if a set of compositions has a potential hyperbolic poset, one can compute better Vandermonde coverings for small $n$ and $s$.

\begin{example}\label{4:ex_6,4}
For $n=6$ and $s=4$ there are $10$ compositions of $6$ into $4$ parts. One can check that out of the $2^{10}$ subsets only $17$ have potential hyperbolic posets. Up to symmetry - we identify $S$ with $\tilde S:=\{(\mu_4,\dots,\mu_1) ~|~ \mu \in S\}$ - we get the $11$ subsets
\begin{gather*}
\{(1, 1, 1, 3), (1, 1, 2, 2), (1, 1, 3, 1)\},\\
\{(1, 1, 3, 1), (1, 2, 2, 1), (1, 3, 1, 1)\},\\
\{(1, 1, 1, 3), (2, 1, 1, 2), (2, 1, 2, 1)\},\\
\{(1, 1, 2, 2), (1, 1, 3, 1), (1, 2, 1, 2), (1, 2, 2, 1)\},\\
\{(1, 2, 1, 2), (1, 2, 2, 1), (2, 1, 1, 2), (2, 1, 2, 1)\},\\
\{(1, 1, 1, 3), (1, 2, 2, 1), (2, 1, 1, 2), (3, 1, 1, 1)\},\\
\{(1, 1, 1, 3), (1, 1, 2, 2), (2, 1, 2, 1), (2, 2, 1, 1)\},\\
\{(1, 1, 1, 3), (1, 1, 3, 1), (2, 1, 1, 2), (2, 2, 1, 1)\},\\
\{(1, 1, 2, 2), (1, 2, 1, 2), (1, 2, 2, 1), (2, 1, 2, 1), (2, 2, 1, 1)\},\\
\{(1, 1, 1, 3), (1, 1, 2, 2), (1, 2, 2, 1), (2, 2, 1, 1), (3, 1, 1, 1)\} \text{ and}\\
\{(1, 1, 3, 1), (1, 2, 1, 2), (1, 2, 2, 1), (2, 1, 1, 2), (2, 2, 1, 1).\}
\end{gather*}
From this we get that $\{(2,2,1,1)\}$ is a $(6,4)$-Vandermonde covering, which is also optimal in this case.
\end{example}

Example \ref{4:ex_6,4} generalizes in the following way:

\begin{proposition}\label{4:thm_covering_(n,n-2)}
$\{(2,2,1,\dots,1)\}$ is a $(n,n-2)$-Vandermonde covering.
\end{proposition}
\begin{proof}
Suppose it is not a Vandermonde covering. Then there would be a hyperbolic slice $\H_s(F)$ with all zero-dimensional strata corresponding to compositions with one entry equal to $3$ and the other entries equal to $1$. By Theorem \ref{2:thm:minmax} all of these compositions correspond to minimal or maximal polynomials in $\H_s(F)$ and therefore $\H_s(F)$ contains at most two zero-dimensional strata. But by Proposition \ref{2:prop:pointormaxd}, $\H_s(f)$ is two-dimensional and thus have at least three extremal points and by Theorem 2.8 in \cite{riener2024linear}, the extremal point of $\H_s(f)$ have at most $s$ distinct roots. This is a contradiction to $\H_s(F)$ having at most two zero dimensional strata
\end{proof}

Since there are $k=\binom{n-1}{s-1}$ compositions of $n$ into $s$ parts, the procedure in Example \ref{4:ex_6,4} becomes too computationally expensive to apply directly when $n$ and $s$ are large since it involves considering $2^k$ subsets. However, we can use some weaker conditions to cut down this big set into a more managable set and that makes it easier to apply our previous method. For example, since we know that every potential hyperbolic poset contains $(1,\dots,1)$, we just have to check all the subsets of compositions with exactly one alternate even and one alternate odd composition. Furthermore, we can apply the bounds stated in Theorem \ref{4:shellable_sphere_g-theorem} and we also know that we need at least $n-s+1$ compositions of length $s$ by the argument in the proof of Theorem \ref{4:thm_covering_(n,n-2)}. This allows computations of all potential hyperbolic posets up to $s,n\leq 9$ on a standard computer with no more than a few hours running time.

\begin{example}\label{4:ex_8,4}
For $n=8$ and $s=4$, we get from Theorem \ref{4:upper_bound} that there is a Vandermonde covering with $3$ partitions and from Theorem \ref{4:recursive_lower_bound} we know that we need at least $1$ partition. By computing all the potential hyperbolic posets we get several Vandermonde coverings with two elements, e.g.
\[\{(3,2,2,1),(4,2,1,1)\},\]
and one can show that there is no Vandermonde covering with only one partition by realizing appropriate potential hyperbolic posets.
\end{example}

\section{Conclusion}

We studied the rich geometric and combinatorial structure of hyperbolic slices. Although we could not show the conjectured polytopality, we were able to establish the weaker result that dual posets of generic hyperbolic posets are combinatorial spheres. We conjectured in \ref{3:conj:comb.sphere} that this is true for general hyperbolic posets. Moreover, we obtained an upper bound theorem for hyperbolic posets from the sphericity of the boundary of the dual posets. We have some computational evidence that this bound is sharp for the number of vertices and maybe also in general. It could be interesting to try to construct and study such "cyclic hyperbolic slices".

It is well known, that every polytope can be obtained as an affine slice of a higher-dimensional simplex. Since a generic hyperbolic poset $L_s(F)$ is a simplex for $s=2$ (see the proof of Theorem 3.10 in \cite{lien2023hyperbolic}), we can see hyperbolic slices as certain affine slices of "hyperbolic simplices". So we ask the following, which is even stronger then the conjectured sharpness of the Upper Bound Theorem:
\begin{question}
    For any $f$-vector of a simple polytope, there are $n,s\in \N$ and a polynomial $F$ such that $\L_s(F)$ has the same $f$-vector. 
\end{question}
In the second part of the paper, we introduced and studied Vandermonde coverings which allow us to strengthen Timofte's degree principle. We showed how to compute better Vandermonde coverings for small $n$ and $s$ by introducing potential hyperbolic posets and conjectured that potential hyperbolic posets are realizable. Such computations might be used to find patterns for Vandermonde coverings for bigger $n$ and $s$.

We suspect that many of our results can be translated to other finite reflection groups, at least to the hyperoctahedral group.

\appendix

\section{Proof of Theorem \ref{2:thm:minmax}}

In this section we prove second part of the following theorem:
\minmax*

We will need some more tools before we get started with the proof and recall that we will let $l=\ell(\mu)>s$ for the proof. For the initial step of the proof we will use Lagrange multipliers so we need a local definition of minimality and maximality:

\begin{definition}
    We call $H=T^n+H_1T^{d-1}+\dots+H_d\in\H_s^\mu(F)$ a \emph{locally minimal} (resp. \emph{locally maximal}) polynomial of the stratum $\H_s^\mu(F)$ if $H_{s+1}\leq G_{s+1}$ (resp. $H_{s+1}\geq G_{s+1}$) for all $G=T^n+G_1T^{d-1}+\dots+G_d\in N$, where $N\subset \H_s^\mu(F)$ is some open neighbourhood of $H$.
\end{definition}

\begin{lemma}\label{2:lem:sdistinct}
    A locally minimal or locally maximal polynomial in $\H_s^\mu(F)$ has at most $s$ distinct roots.
\end{lemma}

\begin{proof}
    Assume $\H_s^\mu(F)$ is at least one-dimensional since the other cases follow from Proposition \ref{2:prop:pointormaxd} and let $l=\ell(\mu)$. By Lemma \ref{2:lem:projhom}, $P^{n-l}:\H_s^\mu(F)\to \mathbb{R}^{l-s}$ is a homeomorphism onto its image which is closed in $\mathbb{R}^{l-s}$. So by Proposition \ref{2:prop:pointormaxd}, the image of the polynomials whose composition is strictly smaller than $\mu$ make up the boundary of $P^{n-l}(\H_s^\mu(F))$. Thus a locally minimal or locally maximal polynomial lies in the relative boundary and therefore has strictly less than $l$ roots and so the statement follows inductively. 
\end{proof}

If $a\in \R^{n-s}$ we will let $B_\epsilon(a)$ denote the open ball about $a$ of radius $\epsilon$.

\begin{lemma}\label{2:lem:loceq}
    A polynomial $H\in\H_s^\mu(F)$ is locally minimal (resp. locally maximal) if and only if it is minimal (resp. maximal).
\end{lemma}

\begin{proof}
    One implication is clear, so suppose $H\in\H_s^\mu(F)$ is locally minimal but not minimal. If $\H^\mu_{s+1}(H)$ is at least one-dimensional then by Proposition \ref{2:prop:pointormaxd}, for any $\epsilon>0$ there is a polynomial $G\in\H_{s+1}^\mu(F)\cap B_\epsilon(H)$ with composition $\mu$. Thus, by Lemma \ref{2:lem:projhom}, there is a $\delta$ with $0<\delta<\epsilon$ such that $P^{n-l}(\H_s^\mu(F))\cap B_\delta(P^{n-l}(G))$ lies in the interior of $P^{n-l}(\H_s^\mu(F))$. So there is a polynomial in $\H_s^\mu(F)\cap B_\epsilon(H)$ whose first free coefficient is smaller than the first free coefficient of $H$ contradicting the local minimality of $H$.

    Thus, by Proposition \ref{2:prop:pointormaxd}, $H^\mu_{s+1}(H)$ must be a point. Since $H^\mu_{s}(F)$ is contractible, there is a path, $\Phi:[0,1]\to H_s^\mu(F)$, where $[0,1]$ is the unit interval, from $H$ to the minimal polynomial. Since $H^\mu_{s+1}(H)$ is a point we may assume that the first free coefficient of $\Phi(y)$ is strictly smaller than the first free coefficient of $H$ for all $y\in(0,1]$. But this is a contradiction since $H$ was assumed to be locally minimal. Thus if $H$ is locally minimal, it must also be minimal. The proof for locally maximal polynomials works analogously.
\end{proof}

It will be useful to work with power sums instead of elementary symmetric polynomials, so we need the follwing lemma:
\begin{lemma}\label{2:lem:powersumiffelem}
    Let $a,b\in\R^n$ and suppose $E_i(a) = E_i(b)$ for all $i\in [s]$, then $P_{s+1}(a)>P_{s+1}(b)$ if and only if $(-1)^{s+1}E_{s+1}(a)<(-1)^{s+1}E_{s+1}(b)$.
\end{lemma}

\begin{proof}
    This is straightforward to show using Newtons identities, see for instance the proof of Proposition 9 in \cite{meguerditchian1992theorem}.
\end{proof}

To prove the second part of Theorem \ref{2:thm:minmax} we will first consider the generic case and do an induction in the partial order. At the end of the section we will extend the statement to the general case.

Note that $\V_s^\mu(F)= \{ x\in\mathbb{R}^l ~|~ \prod_{i=1}^l (T-x_i)^{\mu_i}\in \H_s(F) \}$, so $\V_s^\mu(F)$ corresponds to the polynomials in $\H_s(F)$ whose composition is either $\mu$ or some permutation of $\mu$. In particular, we see that $\V_s^\mu(F)\cap \W_l$ corresponds to the polynomials in $\H_s^\mu(F)$. So by Lemma \ref{2:lem:powersumiffelem}, finding the minimal polynomial in $\H_s^\mu(F)$ corresponds to maximizing $P_{s+1}^\mu(x)$ over the set $\V_s^\mu(F)\cap \W_l.$
However as we will be using Lagrange multipliers we will need to maximize $P_{s+1}^\mu(x)$ over $\V_s^\mu(F)$ instead and make some adjustments.

Before we begin note that by using Newtons inequalities we can rewrite $\V_s^\mu(F)$ as $$\V_s^\mu(F) = \{x\in\mathbb{R}^l~|~P_i^\mu(x) = c_i \ \forall \ i\in [s]\},$$ for some $c_1,\dots,c_s\in \R$. So let $x_1\leq \dots \leq x_{s+1}$ be the roots of a polynomial $H$ in the relative boundary of a one-dimensional stratum $\H_s^\mu(F)$. The determinant of the Jacobian of $(P_1^\mu(x),\dots,P_{s+1}^\mu(x))$ is $c\prod_{i<m}(x_{i}-x_{m})$, for some nonzero constant $c$. Since $x=(x_1,\dots,x_{s+1})$ has $s$ distinct coordinates the determinant vanishes and the vectors $\nabla P_1^\mu(x),\dots,\nabla P_{s+1}^\mu(x)$ are linearly dependant.

Similarly, the determinant of the upper $s\times s$ submatrix of the Jacobian of $(P_1^\mu(x),\dots,P_s^\mu(x))$ does not vanish, so the vectors $\nabla P_1^\mu(x),\dots,\nabla P_s^\mu(x)$ are linearly independant. Thus there are scalars $a_1,\dots,a_s$ such that $\nabla L=0$, where
$$L=P_{s+1}^\mu(x)- \sum_{i=1}^sa_iP_i^\mu(x),$$
and so we say that the pair $(a,x)$, with $a=(a_1,\dots,a_s)$ is a \emph{constrained critical point} (see page 287 in \cite{colley2011vector}).

The gradient of $L$ is $$\nabla L=\nabla P_{s+1}^\mu(x)- \sum_{i=1}^sa_i\nabla P_i^\mu(x) = (Q_1,\dots,Q_l),$$
where $Q_i=\mu_iQ(x_i)=\mu_i((s+1)x_i^s-\sum_{j=1}^sa_jjx_i^{j-1})$. The univariate polynomial $Q(T) = (s+1)T^s-\sum_{j=1}^sa_jjT^{j-1}$ is of degree $s$ and since $Q$ vanishes at $x_i$ for any $i$, then $Q$ have $s$ distinct roots. Also, we see that a point with more than $s$ distinct coordinates cannot be a constrained critical point.

We will look at the Hessian of $L$ to analyse what kind of critical point we have. The Hessian of $L$ is the matrix
$$ HL(x)=\begin{pmatrix}
0 & \cdots & 0 &  m_{1,1} & & \cdots & & m_{1,s+1}\\
\vdots & \ddots & \vdots & \vdots & & \ddots & & \vdots\\
0 & \cdots & 0 &  m_{s,1} & & \cdots & & m_{s,s+1}\\
m_{1,1} & \cdots & m_{s,1} & \mu_1Q'(x_1) & 0 & \cdots & & 0\\
 & & & 0 &  & \ddots &  &  \\
\vdots & \ddots & \vdots & \vdots & \ddots & \ddots & & \vdots\\
 & & & & & &  & 0\\
m_{1,s+1} & \cdots & m_{s,s+1} & 0 &  & \cdots & 0 & \mu_{s+1}Q'(x_{s+1})
\end{pmatrix},
$$
where $m_{i,j} = -i\mu_jx_j^{i-1}$ for $i\in [s]$ and $j\in [s+1]$.

\begin{lemma}\label{lem:princminor}
The determinant of $HL(x)$ is 
$$(-1)^s\sum_{j=1}^{s+1}\bigg{(}b_jQ'(x_j)\prod_{\substack{i,y\in [s+1]\backslash \{j\}:\\ i<y}}(x_{i}-x_{y})^2\bigg{)}$$ for some positive $b_1,\dots,b_{s+1}$.
\end{lemma}

\begin{proof}
    Let $Y = (y_{i,j})=HL(x)$, then the Leibniz formula for the determinant of $Y$ is $$\det(Y) = \sum_{\sigma \in \Sym(2s+1)}\bigg{(}\sign(\sigma)\prod_{i=1}^{2s+1}y_{i,\sigma(i)}\bigg{)}.$$
    If $\sign(\sigma)\prod_{i=1}^{2s+1}y_{i,\sigma(i)}$ is nonzero, then $$\prod_{i=1}^{2s+1}y_{i,\sigma(j)}= m_{1,r_1}\cdots m_{s,r_s}\cdot \mu_jQ'(X_j) \cdot m_{1,k_1}\cdots m_{s,k_s},$$ where $\{r_1,..,r_s,j\}=\{k_1,..,k_s,j\}=[s+1]$.
    
    Thus $\sigma(s+j)=s+j$ for some $j\in [s+1]$ and then $\sigma([s])=\{s+1,\dots,2s+1\}\backslash \{s+j\}$ and $\sigma(\{s+1,\dots,2s+1\}\backslash \{s+j\})=[s]$. Therefore we have 
    $$\sigma = \prod_{i\in [s]}(i,k_i)\prod_{i\in [s]}(i,s+r_i) = \prod_{i\in [s]}(i,k_i)\prod_{i\in [s]}(r_i,s+r_i)\prod_{i\in [s]}(i,r_i),$$
    where $(i,m)$ denotes a transposition, and so
    $$\sign(\sigma)=(-1)^s\sign\bigg{(}\prod_{i\in [s]}(i,k_i)\bigg{)}\sign\bigg{(}\prod_{i\in[s]}(i,r_i)\bigg{)}.$$
    
    Thus the determinant of $HL(x)$ is $$(-1)^s\sum_{j=1}^{s+1}\bigg{(}\mu_jQ'(x_j)\sum_{\substack{(r_1,\dots,r_s):\\ \{r_1,..,r_s,j\}=[s+1]}}\sign\bigg{(}\prod_{i\in [s]}(i,r_i)\bigg{)}\prod_{i\in [s]} m_{i,r_i}\cdot $$
    $$ \sum_{\substack{(k_1,\dots,k_s):\\ \{k_1,..,k_s,j\}=[s+1]}}\sign\bigg{(}\prod_{l\in [s]}(l,k_l)\bigg{)}\prod_{l\in [s]} l_{l,k_l}\bigg{)}.$$

    Both of the last two sums are equal to the determinant of the $s\times s$ submatrix $(m_{i,y})_{i\in [s], y\in[s+1]\backslash \{j\}}$. This is a Vandermonde matrix with weighted columns and rows thus its determinant is $$-\bigg{(}\prod_{i\in [s+1]\backslash \{j\}}i\mu_{i}\bigg{)}\bigg{(}\prod_{\substack{i,y\in [s+1]\backslash \{j\}:\\ i<y}}(x_{i}-x_{y})\bigg{)}.$$
    So if we let $b_j=\mu_j(\prod_{i\in [s+1]\backslash \{j\}}i\mu_{i})^2$, the statement follows.
\end{proof}

\begin{proposition}\label{2:prop:indstart}
Let $H_s^u(f)$ be generic and one-dimensional. Then $H\in\H_s^\mu(F)$ is the minimal (resp. maximal) polynomial if and only if $\ell(c(H)) =s$ and $c(H)/\mu$ is alternate odd (resp. even).
\end{proposition}

\begin{proof}
Let $H=\prod_{i=1}^{s+1}(T-x_i)^{\mu_i}$ be in the relative boundary of $\H_s^\mu(F)$. Then $H$ has $s$ distinct roots and is either the maximal or the minimal polynomial. The Hessian criterion from \cite{colley2011vector} (chapter 4, page 288) says that $x$ is a local minimum (resp. maximum) of $P_{s+1}(x_u)$ on $\V_s^\mu(F)$ if and only if $(-1)^s\det(HL(x))$ is positive (resp. negative). So by Lemma \ref{2:lem:powersumiffelem}, $x$ is a local minimum (resp. maximum) of $(-1)^{s+1}E_{s+1}(x_u)$ on $\V_s^\mu(F)$ if and only if $(-1)^s\det(HL(x))$ is negative (resp. positive).

By Lemma \ref{lem:princminor},
$$(-1)^s\det(HL(x))=\sum_{j=1}^{s+1}\bigg{(}b_jQ'(x_j)\bigg{(}\prod_{\substack{i,y\in [s+1]\backslash \{j\}:\\ i<y}}(x_{i}-x_{y})\bigg{)}^2\bigg{)},$$
where the Vandermonde determinant is zero whenever two coordinates of $x$ with indices in $[s+1]\backslash \{j\}$ are equal. That is, the only nonzero terms are when $x_j$ is the repeated coordinate of $x$. There are two such indices so let them be $k$ and $k+1$. Then we have 
$$(-1)^s\det(HL(x))=b_kQ'(x_k)\prod_{\substack{i,y\in [s+1]\backslash \{k\}:\\ i<y}}(x_{i}-x_{y})^2+$$
$$b_{k+1}Q'(x_{k+1})\prod_{\substack{i,y\in [s+1]\backslash \{k+1\}:\\ i<y}}(x_{i}-x_{y})^2.$$

Since $Q'(x_k)$ and $Q'(x_{k+1})$ have the same sign and all the other factors are positive we have that $\sign((-1)^s\det(HL(x))) = \sign(Q'(x_k)).$ Since $Q$ has no repeated roots, $Q'(x_k)\neq 0$. As the roots of $Q'$ interlace the roots of $Q$ and the leading coefficient of $Q$ is positive, the sign of $Q'(x_k)$ is negative if and only if $x_k$ is the second largest, fourth largest\dots, or $(2j)^{th}$ largest coordinate of $x$. That is, if and only if $c(H)/\mu$ is alternate even. Similarly, the sign of $Q'(x_k)$ is positive if and only if $x_k$ is the largest, third largest,\dots, or $(2j+1)^{th}$ largest coordinate of $x$. That is, if and only if $c(H)/\mu$ is alternate odd.

Note also that for a sufficiently small $\epsilon>0$, the ball $B_\epsilon(x)$ meets the two open sets $\{y\in\R^{s+1}~|~y_1<y_2<\dots<y_{s+1}\}$ and $\{y\in\R^{s+1}~|~y_1<\dots<y_{k-1}<y_{k+1}<y_k<y_{k+2}<\dots<y_{s+1}\}$ but no other set of that form. Thus $H$ is either the locally minimal polynomial or the locally maximal polynomial for both the strata $\H_s^\mu(F)$ and $\H_s^\nu(F)$, where $\nu=(\mu_1,\dots,\mu_{k-1},\mu_{k+1},\mu_k,\mu_{k+2},\dots,\mu_{s+1})$.

In particular, the alternate odd- or evenness of $c(H)/\mu$ along with the length restriction $\ell(c(H))=s$, are sufficient and necessary conditions for $H$ being locally minimal or locally maximal in $\H_s^\mu(F)$ and by Lemma \ref{2:lem:loceq} this is equivalent to $H$ being minimal or maximal in $\H_s^\mu(F)$.
\end{proof}

Having settled the initial step of our induction, we need to establish some tools for the inductive step. Firstly we need something on the combinatorial side and we start by rephrasing Definition \ref{2:def:mod}: so if $\lambda \leq \mu$ and $r=\ell(\lambda)$, then there is an increasing sequence of integers $n_0,\dots,n_r$, with $n_0=0$ and $n_r=l=\ell(\mu)$, such that $\lambda_i=\sum_{j<n_{i-1}}^{n_i}\mu_j$ for all $i\in [r]$. Then the composition $\lambda/\mu$ is the composition of $l$ whose parts are $(\lambda/\mu)_i=n_i-n_{i-1}$.
\begin{lemma}\label{2:lem:smalleriff}
Let $\lambda,\gamma<\mu$ be compositions of $d$, then we have $\lambda/\mu<\gamma/\mu$ if and only if $\lambda<\gamma$ and in this case we have that $\lambda/\gamma=\frac{\lambda/\mu}{\gamma/\mu}$. 
\end{lemma}

\begin{proof}
We continue with the notation above and similarly as for $\lambda$ we have that if $\gamma$ is of length $k$, then there is an increasing sequence of integers $m_0<\dots<m_k$ with $m_0 = 0$ and $m_k=l$ such that $\gamma_i=\sum_{j>m_{i-1}}^{m_i}\mu_j, \ \forall \ i \in [k]$. So $\lambda/\mu = (n_1-n_0,\dots,n_r-n_{r-1})$ and $\gamma/\mu = (m_1-m_0,\dots,m_k-m_{k-1})$ are two compositions of $l$.

If $\lambda<\gamma$, there is an increasing sequence of integers $z_0<\dots<z_r$ with $z_0 = 0$ and $z_r=k$ such that $\lambda_i=\sum_{j>z_{i-1}}^{z_i}\gamma_j \ \forall \ i \in [r]$.  Thus $$\sum_{j>n_{i-1}}^{n_i}\mu_j=\lambda_i=\sum_{j>z_{i-1}}^{z_i}\bigg{(}\sum_{y>m_{j-1}}^{m_j}\mu_y\bigg{)} = \sum_{y>m_{z_{i-1}}}^{m_{z_i}}\mu_y, \ \forall \ i \in [r],$$
and since $m_0=n_0$ we have $m_{z_i} = n_i$ and $m_{z_{i-1}} = n_{i-1}$. Thus $$(\lambda/\mu)_i=n_i-n_{i-1}= m_{z_i}-m_{z_{i-1}} =$$
$$ m_{z_i}-m_{z_{i}-1}+m_{z_i-1}-m_{z_{i}-2}+\dots+m_{z_{i-1}+1}-m_{z_{i-1}} = \sum_{j>z_{i-1}}^{z_i}(\gamma/\mu)_j$$ and so $\lambda/\mu<\gamma/\mu$.

Conversely, if $\lambda/\mu<\gamma/\mu$, then there is an increasing sequence of integers $y_0<\dots<y_r$ with $y_0 = 0$ and $y_r=k$ such that $$(\lambda/\mu)_i=\sum_{j>y_{i-1}}^{y_i}(\gamma/\mu)_j, \ \forall \ i \in [r].$$ Thus we have $$n_i-n_{i-1} = (\lambda/\mu)_i=\sum_{j>y_{i-1}}^{y_i}(m_j-m_{j-1}) =m_{y_i}-m_{y_{i-1}},$$
and since $n_0=0=m_0=m_{z_0}$, we have $n_i=m_{y_i} \ \forall \ i \in [r]$. Thus $$\lambda_i=\sum_{j>n_{i-1}}^{n_i}\mu_j=\sum_{j>m_{y_{i-1}}}^{m_{y_i}}\mu_j=\sum_{j>y_{i-1}}^{y_i}\gamma_j \ \forall \ i \in [r]$$ and so $\lambda<\gamma.$

Lastly, since $n_i=m_{y_i}$ and $n_i=m_{z_i}$, we have $m_{y_i}=m_{z_i}$. Since the indices $m_0,\dots,m_k$ are distinct we have $y_i=z_i$. Thus $$\bigg{(}\frac{\lambda/\mu}{\gamma/\mu}\bigg{)}_i = y_i-y_{i-1} = z_i-z_{i-1} = (\lambda/\gamma)_i \ \forall \ i\in[r]$$ and so we have $\lambda/\gamma=\frac{\lambda/\mu}{\gamma/\mu}$.
\end{proof}

Next we need to look closer at the projection introduced in the beginning of the appendix. It should be noted that the following discussion and lemma is analogous to the approach in \cite{kostov1989geometric}, where the image of the power sums are studied instead of the elementary symmetric polynomials.

By Lemma \ref{2:lem:projhom}, $\H_s^\mu(F)$ is homeomorphic to $P^{n-l}(\H_s^\mu(F))\subset \R^{l-s}$ and thus by Proposition \ref{2:prop:pointormaxd}, $M:=P^{n-l}(\H_s^\mu(F))$ is full-dimensional when $\H_s^\mu(F)$ is neither empty nor a single polynomial. Let $\pi:M\to \R^{l-s-1}$ be the projection given by $(x_1,\dots,x_{l-s})\mapsto (x_1,\dots,x_{l-s-1})$, then for $H\in\H_s^\mu(F)$, the fibre $\pi^{-1}(\pi(P^{n-l}(H)))$ equals $P^{n-l}(\H_{l-1}^\mu(H))$. This fibre is by Proposition \ref{2:prop:pointormaxd}, either the point $P^{n-l}(H)$, in which case it must lie on the boundary of $M$, or it is an interval. And if it is an interval, then its endpoints must lie on the boundary of $M$ and its relative interior lies in the interior of $M$.

Thus the boundary of $M$ can be written as the union of a "lower" and an "upper" part, $L\cup U$, where $$L=\{(x_1,\dots,x_{l-s})\in M~|~x_{l-s}\leq y_{l-s} \ \forall \ (y_1,\dots,y_{l-s})\in \pi^{-1}(\pi(x))\},$$
and $$U=\{(x_1,\dots,x_{l-s})\in M~|~x_{l-s}\geq y_{l-s} \ \forall \ (y_1,\dots,y_{l-s})\in \pi^{-1}(\pi(x))\}.$$

\begin{lemma}\label{2:lem:contupperlower}
    The sets $L$ and $U$ are closed.
\end{lemma}

\begin{proof}
    We just show that $U$ is closed since the proof for $L$ is analogous. So suppose $P^{n-l}(Q)$ is in the closure of $U$ but not in $U$. By Lemma \ref{2:lem:projhom}, the boundary of $P^{n-l}(\H_s^\mu(F))$ is closed and thus $P^{n-l}(Q)\in L$. Thus $\pi^{-1}(\pi(P^{n-l}(Q)))$ is an interval whose relative interior lies in the interior of $P^{n-l}(\H_s^\mu(F))$. Let $P^{n-l}(G)$ be one of those relative interior points and let $\epsilon>0$ be such that $B_\epsilon(P^{n-l}(G))\subset P^{n-l}(\H_s^\mu(F))$.

    For any $P^{n-l}(H)\in B_\epsilon(P^{n-l}(G))$,  the point $\pi^{-1}(\pi(P^{n-l}(H)))\cap L$ lies below $B_\epsilon(P^{n-l}(G))$. Thus the distance between $P^{n-l}(Q)$ and any point in $U$ is at least as large as $\epsilon/2$. Thus $P^{n-l}(Q)$ cannot be in the closure of $U$ which is a contradiction and so $P^{n-l}(Q)$ must lie in $U$.
\end{proof}

\maxiffmax

\begin{proof}
    One implication is clear, so we just have to show that if for all compositions $\nu$, with $H\in \H_s^\nu(F)\subsetneq \H_s^\mu(F)$, we have that $H$ is minimal in $\H_s^\nu(F)$, then $H$ is minimal in $\H_s^\mu(F)$. We assume $\H_s^\mu(F)$ is $(l-s)$-dimensional since the statement is clear when it is just a point. Also, the argument for maximal polynomials is analogous so we just prove it for minimal polynomials.

    Suppose $H$ is not minimal in $\H_s^\mu(F)$, then by Lemma \ref{2:lem:loceq} it is not locally minimal. So for any $i\in\mathbb{N}$, $B_{1/i}(H)\cap \H_s^\mu(F)$ contains a polynomial $G_i$ whose first free coefficient is smaller than the first free coefficient of $H$.

    Without loss of generality assume $P^{n-l}(H)$ lies in the upper part of the boundary of $M$. Then for each fibre $\pi^{-1}(\pi(P^{n-l}(G_i)))$, let $P^{n-l}(Q_i)$ be the point in the upper part of the boundary of $M$. Since the upper part is compact by Lemma \ref{2:lem:compcontr} and Lemma \ref{2:lem:projhom}, $(P^{n-l}(Q_i))$ converges to a point in the upper part which is by design $P^{n-l}(H)$.

    As there are finitely many compositions, there is an infinite subsequence of $(P^{n-l}(Q_i))$, where all the $Q_i$'s have the same composition $\lambda\neq \mu$, that converges to $P^{n-l}(H)$. By Proposition \ref{2:prop:pointormaxd} and Lemma \ref{2:lem:projhom}, the image $P^{n-l}(\H_s^\lambda(F))$ is the closure of its relative interior which consists of the images of the polynomials with composition $\lambda$. Thus $H\in\H_s^\lambda(F)$ and it is by construction not the minimal polynomial. This is a contradiction and so $H$ must be minimal in $\H_s^\mu(F)$.
\end{proof}

\begin{lemma}\label{2:lem:diamond}
Let $l=\ell(\mu)\geq s+2$ and let $H\in\H_s^\mu(F)$ have $s$ distinct roots. Then there are two polynomials with distinct compositions, $\gamma$ and $\nu$, in $\H_s^\mu(F)$ of length $\ell(\mu)-1$ and with $c(H)<\gamma,\nu$.
\end{lemma}

\begin{proof}
    Let $\lambda=c(H)$, then since $l\geq s+2$, $\ell(\lambda)=s$ and $\lambda<\mu$ one must replace at least two of the commas in $\mu$ with plus signs to obtain $\lambda$. So let $j\neq i$ be two indices such that $$\gamma=(\mu_1,\dots,\mu_{j-1},\mu_j+\mu_{j+1},\mu_{j+2},\dots,\mu_{l})$$ and $$\nu=(\mu_1,\dots,\mu_{i-1},\mu_i+\mu_{i+1},\mu_{i+2},\dots,\mu_{l})$$ are both greater than $\lambda$. By Proposition \ref{2:prop:pointormaxd} both of these compositions must occur in $\H_s^\mu(F)$.
\end{proof}

\begin{proposition}\label{2:prop:generic}
Let $H_s^\mu(F)$ be of $(l-s)$-dimensional and generic. Then $H\in\H_s^\mu(F)$ is the minimal (resp. maximal) polynomial if and only if $\ell(c(H))=s$ and $c(H)/\mu$ is alternate odd (resp. even).
\end{proposition}

\begin{proof}
We prove this by induction in the poset of strata of $\H_s^\mu(F)$. The initial step is when $l= s+1$ and is covered by Proposition \ref{2:prop:indstart}. Next, we assume the statement is true for the strata of dimension $l-s-1\geq 1$ and we show that it is true when the stratum is $(l-s)$-dimensional. We will just show the proof for minimal polynomials as the proof for maximal polynomials is analogous.

Let $\lambda=c(H)$ and suppose $\lambda/\mu$ is alternate odd and that $\ell(\lambda)=s$. Let $\gamma$ be any composition with $\lambda<\gamma<\mu$ such that $\H_s^\gamma(F)$ is at least one-dimensional. By Lemma \ref{2:lem:smalleriff} we have that $\lambda/\gamma=\frac{\lambda/\mu}{\gamma/\mu}$. Note that the $i^{th}$ part of $\lambda/\gamma$ is equal to the $i^{th}$ part of $\lambda/\mu$ minus some integer, thus $\lambda/\gamma$ is alternate odd since $\lambda/\mu$ is. So by the induction hypothesis, $H$ is the minimal polynomial of $\H_s^\gamma(F)$. And so by Lemma \ref{2:lem:maxiffmax}, $H$ is the minimal polynomial of $\H_s^\mu(F)$.

For the reverse statement, let $H$ be the minimal polynomial. Then by Lemma \ref{2:lem:sdistinct}, $H$ has $s$ distinct roots. Since $H_s^\mu(f)$ is at least two-dimensional, then by Lemma \ref{2:lem:diamond}, there occurs at least two distinct compositions, $\gamma$ and $\nu$ in $\H_s^\mu(F)$, of length $l-1$ and where $\lambda<\gamma,\nu$. By Proposition \ref{2:prop:pointormaxd}, the strata $\H_s^\gamma(F)$ and $\H_s^\nu(F)$ are $(l-s-1)$-dimensional.

By Lemma \ref{2:lem:maxiffmax} and the induction hypothesis this means that $\lambda/\gamma$ and $\lambda/\nu$ are alternate odd compositions. Since $\gamma$ and $\nu$ are of length $l-1$, there are two indices $j\neq i$ such that $$\lambda=(\mu_1,\dots,\mu_{j-1},\mu_j+\mu_{j+1},\mu_{j+2},\dots,\mu_l)$$ and
$$\nu=(\mu_1,\dots,\mu_{i-1},\mu_i+\mu_{i+1},\mu_{i+2},\dots,\mu_l).$$
Thus $\gamma/\mu = (1,\dots,1,2,1,\dots,1)$, where the index $2$ is in the $j^{th}$ position and $\nu/u = (1,\dots,1,2,1,\dots,1)$, where the index $2$ is in the $i^{th}$ position.

Since $\lambda/\gamma=\frac{\lambda/\mu}{\gamma/\mu}$ and $\lambda/\nu=\frac{\lambda/\mu}{\nu/\mu}$, we have that $$\lambda/\gamma=((\lambda/\mu)_1,\dots.,(\lambda/\mu)_{j-1},(\lambda/\mu)_j-1,(\lambda/\mu)_{j+1},\dots,(\lambda/\mu)_s)$$ and that
$$\lambda/\nu=((\lambda/\mu)_1,\dots.,(\lambda/\mu)_{i-1},(\lambda/\mu)_i-1,(\lambda/\mu)_{i+1},\dots,(\lambda/\mu)_s).$$ Since $j\neq i$ then $\lambda/\gamma\neq \lambda/\nu$ and since both compositions are alternate odd then so must $\lambda/\mu$ be.
\end{proof}

Now that we have established the second part of Theorem \ref{2:thm:minmax} for the generic case we will extend it to the non-generic cases. Note that we will be using Proposition \ref{3:prop:subdiv} in the following two proofs, but as that proposition only requires the first part of Theorem \ref{2:thm:minmax}, there are no circular arguments.

\begin{lemma}\label{2:lem:nongen}
If $H\in\H_s^\mu(F)$ and $c(H)<\gamma$ for some $\gamma$, of length $s$, such that $\gamma/\mu$ is alternate odd (resp. even), then $H$ is minimal (resp. maximal).
\end{lemma}

\begin{proof}
Again, we just show the statement for minimal polynomials. If $\H_s^\mu(F)$ is just a point, the statement is clear so by Proposition \ref{2:prop:pointormaxd}, we may assume it is $(l-s)$-dimensional, where $l>s$. Thus we may also assume $F$ has no repeated roots. Suppose $H$ is not minimal, then by Lemma \ref{2:lem:loceq}, $H$ is not locally minimal.

Thus for any $\delta>0$, $B_\delta(H)\cap\H_s^\mu(F)$ contains a polynomial contains a polynomial $Q$ whose first free coefficient is $r\in \R_{>0}$ smaller than the first free coefficient of $H$. By Proposition \ref{2:prop:pointormaxd}, $\H_s^\mu(F)$ is the closure of its relative interior, so we may assume $c(Q)=\mu$. So by Lemma \ref{2:lem:projhom}, $P^{n-l}(Q)$ is an interior point of $P^{n-l}(\H_0^\mu(F))$, and there is therefore an $\epsilon$ with $0<\epsilon<r/2$ such that $B_\epsilon(P^{n-l}(Q))\subset P^{n-l}(\H_0^\mu(F))$.

All compositions occur in $\H_0(F)$ and since $\H_0^\gamma(F)$ is the closure of its relative interior then $B_\epsilon(P^{n-l}(H))\cap P^{n-l}(\H_0^\gamma(F))$ contains a point, $P^{n-l}(G)$, where $c(G)=\gamma$. The intersection $P^{n-l}(\H_s^\mu(G))\cap B_\epsilon(P^{n-l}(Q))$ is nonempty since the first $s+1$ coefficients of $Q$ equals the first $s+1$ coefficients of $H$ and $P^{n-l}(G)\in B_\epsilon(P^{n-l}(H)) $. Thus there is a polynomial from $B_\epsilon(Q)$ in $\H_s^\mu(G)$.

By Proposition \ref{3:prop:subdiv}, we may assume $\H_s^\mu(G)$ is generic which, by Proposition \ref{2:prop:generic}, means that $G$ must be the minimal polynomial of $\H_s^\mu(G)$. However the first free coefficient of any polynomial from $B_\epsilon(Q)$ is smaller than the first free coefficient of $H$ minus $r/2$ and the first free coefficient of $G$ is at least as large as the first free coefficient of $H$ minus $r/2$. This is a contradiction and so $H$ must be minimal in $\H_s^\mu(F)$.
\end{proof}

\begin{lemma}\label{2:lem:nongen2}
If $H\in\H_s^\mu(F)$ and $c(H)\not< \nu$ for any $\nu$, of length $s$, such that $\nu/\mu$ is alternate odd (resp. even), then $H$ is not minimal (resp. maximal).
\end{lemma}

\begin{proof}
Again, we just show the statement for alternate odd compositions. If $s=2$, then by the theorem in \cite{meguerditchian1992theorem} either $F$ has only one distinct root and $\H_2(F)=\{F\}$ or $\H_2(F)$ contains no polynomials with strictly less than two distinct roots. By assumption we are not in the former case and so $\H_2(F)$ is generic and thus the statement follows from Proposition \ref{2:prop:generic}.

Next we treat the cases when $s\geq 3$ and by the previous paragraph we have that $H_2(F)$ is generic and all but the composition $(n)$ occurs. By Proposition \ref{2:prop:pointormaxd}, $\H_2^\mu(F)$ is the closure of its relative interior, so for any integer $i\geq 1$ there is a polynomial $G_i\in B_{1/i}(H)\cap \H_2^\mu(F)$ with composition $\mu$. Due to Proposition \ref{3:prop:subdiv}, we may assume $\H_s^\mu(G_i)$ is generic. Thus, by Proposition \ref{2:prop:generic}, the composition, $\nu$, of the minimal polynomial in $\H_s^\mu(G_i)$ is such that $\nu/\mu$ is alternate odd.

Since there are finitely many compositions with this property, there is one such $\nu$ such that for infinitely many $i$, the minimal polynomial of $\H_s^\mu(G_i)$ has composition $\nu$. So we may assume that for all $i\geq 1$, the minimal polynomial, $Q_i$, of $\H_s^\mu(G_i)$ has the same composition $\nu$. Since $(1/i)_{i\geq 1}$ converges to zero and $\H_2^\mu(F)$ is compact, the sequence $(G_i)_{i\geq 1}$ converges. Similarly, since $\H_2^\mu(F)$ is sequentially compact, an infinite subsequence of $(Q_i)_{i\geq 1}$ converges and so for notations sake we will assume this is the sequence $(Q_i)_{i\geq 1}$.

The limit of $(G_i)_{i\geq 1}$ is $H$ and since the first $s+1$ coefficients of $Q_i$ is equal to the first coefficients of $G_i$, the limit, $Q$, of $(Q_i)_{i\geq 1}$ also lies in $\H_s^\mu(F)$. Since $\H_2^\nu(F)$ is the closure of its relative interior and $c(Q_i)=\nu$ for all $i$, then $c(Q)\leq \nu$ and thus by Lemma \ref{2:lem:nongen}, $Q$ is the minimal polynomial of $\H_s^\mu(F)$. Since $c(H)$ is not smaller than a composition $\gamma$ such that $\gamma/\mu$ is alternate odd, then $c(H)\not<\nu$ and thus $H\neq Q$. So $H$ is not the minimal polynomial of $\H_s^\mu(F)$.
\end{proof}

Proposition \ref{2:prop:generic} proves the second part of Theorem \ref{2:thm:minmax} for the generic cases and the combination of Lemma \ref{2:lem:nongen} and Lemma \ref{2:lem:nongen2} proves it for the non-generic cases. And since we proved the first part of Theorem \ref{2:thm:minmax} for all cases in Section \ref{sec:strat}, our work is done. 

\bibliographystyle{alpha}
\bibliography{bibliography}

\end{document}